\documentclass[12pt,a4]{article}
\usepackage{lineno}

\usepackage{geometry}
\usepackage{amsmath}
\usepackage{amsthm}
\usepackage{amssymb}
\usepackage{graphicx}
\usepackage{authblk}

\usepackage{hyperref}
\usepackage{xcolor}
\usepackage{esvect}
\usepackage[shortlabels]{enumitem}
\usepackage{tabularx}
\usepackage{multirow}
\usepackage{booktabs}
\usepackage[caption=false]{subfig}
\usepackage{mlmath}

\usepackage{caption}
\newtheorem{innercustomgeneric}{\customgenericname}
\providecommand{\customgenericname}{}
\newcommand{\newcustomtheorem}[2]{%
  \newenvironment{#1}[1]
  {%
   \renewcommand\customgenericname{#2}%
   \renewcommand\theinnercustomgeneric{##1}%
   \innercustomgeneric
  }
  {\endinnercustomgeneric}
}

\newcustomtheorem{customthm}{Theorem}
\newcustomtheorem{customlemma}{Lemma}
\providecommand{\keywords}[1]
{
  \small	
  \textbf{\textit{Keywords---}} #1
}

\counterwithin{figure}{subsection}
\numberwithin{figure}{subsection}

\begin{document}
\numberwithin{equation}{section}
\title{Numerical Stability for Differential Equations with Memory}
\author{
Guihong Wang\textsuperscript{\rm 1,2}, 
Yuqing Li\textsuperscript{\rm 1,2},  
Tao Luo\textsuperscript{\rm 1,2,3,4,5}\thanks{Corresponding author: luotao41@sjtu.edu.cn.},
Zheng Ma\textsuperscript{\rm 1,2,3,4}\thanks{Corresponding author: zhengma@sjtu.edu.cn.},
Nung Kwan Yip\textsuperscript{\rm 6}, 
Guang Lin\textsuperscript{\rm 6,7}.\\
\textsuperscript{\rm 1}  School of Mathematical Sciences, Shanghai Jiao Tong University \\
\textsuperscript{\rm 2}  CMA-Shanghai, Shanghai Jiao Tong University\\
\textsuperscript{\rm 3} Institute of Natural Sciences, MOE-LSC, Shanghai Jiao Tong University \\
\textsuperscript{\rm 4} Qing Yuan Research Institute, Shanghai Jiao Tong University\\
\textsuperscript{\rm 5} Shanghai Artificial Intelligence Laboratory\\
\textsuperscript{\rm 6} Department of Mathematics, Purdue University\\
\textsuperscript{\rm 7} School of Mechanical Engineering, Purdue University\\

\{theodore123, liyuqing\underline{~}551,luotao41,zhengma\}@sjtu.edu.cn, \{yipn, guanglin\}@purdue.edu.
}
\date{\today}
\maketitle
\begin{abstract}
 In this note, we systematically investigate linear multi-step methods for differential equations  with memory.  In particular, we focus on the numerical stability for multi-step methods. According to this investigation, we give some sufficient conditions for the stability and convergence of some common multi-step methods, and accordingly, a  notion of A-stability for differential equations with memory. Finally, we carry out the computational performance of our theory through numerical examples.
\end{abstract}
\section{Introduction}\label{Section...Introduction}

The past few decades have been witnessing a strong interest among physicists, engineers and mathematicians for the theory and numerical modeling of ordinary differential equations~(\textbf{ODEs}). For ODEs, some are solved analytically; see, for example the excellent book by V.I. Arnold~\cite{Arnold1998Ordinary}, and the references therein. However, only some special cases of these equations can be solved analytically, so we shall turn to numerical solutions as an  alternative.
Sophisticated methods have been developed recently for the numerical solution of ODEs, to name a few, the most commonly used Euler methods, the elegant Runge-Kutta and extrapolation methods, and multistep and general multi-value methods, all of which can be found in the well-known book by Hairer et al.~\cite{Hairer1987solving}.

    According to~\cite{Arino2007Delay}, Picard  (1908) emphasized the importance of the consideration of hereditary effects in the modeling of physical systems. Careful studies of the real world compel us  that, reluctantly, the rate of change of physical systems depends not only their present state, but also on their past history~\cite{Bellman1963Differential}. Development of the theory of delay differential equations~(\cite{Arino2007Delay,smith2011introduction}) gained much momentum  after that. DDEs belong to the class of functional differential equations, which are infnite dimensional, as opposed  to ODEs. The study of DDEs is also popular and stability of some linear method for DDEs have been discussed in \cite{TORELLI198915,wang2009nonlinear}. However,  the hypothesis that a physical system depends on the time lagged solution at a given point is not realistic, and one should rather extend its dependence over a longer period of time instead of a instant~\cite{Hairer1987solving}. Volterra (1909), (1928) discussed the integro-differential equations that model viscoelasticity and  he wrote a book on the role of hereditary effects on models for the interaction of species (1931). To illustrate the difference, DDEs are equations with ``time lags'', such as
\begin{equation}\label{eq...Introduction...DDE}
    \dot{x}(t) = f(t,x(t),x(t-\tau)),
\end{equation}
whilst an integro-differential equation reads
\begin{equation}\label{eq...Introduction...IntegralDifferential}
    \dot{x}(t)=f\left(t,x, \int_{t-\tau}^t K(t,\xi,x(\xi)) \diff \xi \right),
\end{equation}
where the memory term comes into effect and $\tau$ is the duration of the memory. 
Moreover, instead of integro-differential equations that are only concerned with the local ``time lags", we focus on the  ODEs with memory which contains all the global states. Generally, an ODE with memory in this note reads
\begin{equation}\label{eq...introduction...TargetEquation}
    \dot{\vx}(t)=\vf(\vx(t),t)+\int_0^t \vg(\vx(s),s,t) \diff s, \quad \vx(0)=\vx_0,
\end{equation}
where $\vx:\overline{\mathbb{R}}_+\rightarrow \mathbb{R}^d$, $\vf:\mathbb{R}^d\times\overline{\mathbb{R}}_+\rightarrow\mathbb{R}^d$, $\vg:\mathbb{R}^d\times\overline{\mathbb{R}}_+\times\overline{\mathbb{R}}_+\rightarrow\mathbb{R}^{d}$. 
In fact, the memory term in \eqref{eq...introduction...TargetEquation} is quiet common in various fields. A popular example is the time-fractional differential equation \cite{hilfer2000applications,podlubny1999introduction} which is used to describe the subdiffusion process.  Numerical schemes and the analysis for these differential equations with memory have been estabished in \cite{liao2018second,liao2019discrete,liao2018sharp,liu2022unconditionally}. 

In this note, we aim to develop new tools for the analysis of linear multi-step methods~(\textbf{LMMs}) designed  for the initial value problem~\eqref{eq...introduction...TargetEquation}.
Truncation errors and numerical stability of classical LMMs, i.e., LMMs designed for normal ODEs,  are well understood. Due to the existence of the memory term, however,  one may find the implementation of LMMs on~\eqref{eq...introduction...TargetEquation} challenging.  A nature resolution is the  introduction of  quadratures, a broad family of algorithms for calculating the numerical value of a definite integral, in purpose of the approximation of the memory integration.  Hence, LMMs  for ODEs with memory are numerical methods comprising characteristics from classical LMMs and quadratures. However, such LMMs are far from being well understood from a numerical analysis perspective. More importantly, the numerical stability might not be guaranteed due to the accumulation of error aroused from the memory integration. In this note, we present our study on LMMs designed  for ODEs with memory, in particular,  providing  a definition of stability, which is consistent to the classical ODE case. We believe our methods can be extended to nonlinear methods such as Runge-Kutta.

We have to clarify that  Theorem~\ref{thm...convergence} which describes the convergence for general linear multistep method in Section~\ref{subsection...convergence} has been proved in \cite[Chapter 11]{linz1985analytical} where  \eqref{eq...introduction...TargetEquation} is called Volterra-type integro-differential equation.  Our main results are essentially equivalent to those in~\cite{linz1985analytical}.

\section{Preliminaries}\label{Section...Preliminaries}
\subsection{Notations}\label{subsection...Notations}
We begin this section by introducing some notations that will be used in the rest of this note. 
We set $h$  as the constant step size, then the $k$-th grid point  $t_k:=kh$ for all $k\in\mathbb{N}$.
We let $[k]=\{1,2, \ldots, k\}$, and  we use $\fO(\cdot)$ and $o(\cdot)$ for the standard Big-O and Small-O notations. 
Finally, we denote vector $L^2$ norm as $\Norm{\cdot}_2$, vector or function $L_{\infty}$ norm as $\Norm{\cdot}_{\infty}$, matrix spectral~(operator) norm as $\Norm{\cdot}_{2\to 2}$, and  matrix infinity norm as $\Norm{\cdot}_{\infty\to\infty}$.
\subsection{Linear Multistep Methods}\label{Subsection...Description of the scheme}
\begin{defi}\label{Definition....NumericalScheme}
A \textbf{linear $q$-step method} for  the initial value problem~\eqref{eq...introduction...TargetEquation} reads
\begin{equation}\label{eq...preliminary...qstep}
    \vx_n=\sum_{i=1}^{q}\alpha_{q-i}\vx_{n-i}+h\sum_{i=0}^{q}\beta_{q-i}\vf_{n-i}+h\sum_{i=0}^{n}w_{n,i}\vg_{n,i}.
\end{equation}
  In the above formula, the coefficients $\left\{\alpha_k\right\}_{k=0}^{q-1}$ and $\left\{\beta_k\right\}_{k=0}^{q}$ satisfy $\abs{\alpha_0}+\abs{\beta_0}>0$, and  $\vf_{n}:=\vf(\vx_n,t_n)$,  $\vg_{n,i}:=\vg(\vx_{i},t_i,t_n)$. In terms of LMMs, if $\beta_q=0$, we say the method is \textbf{explicit}. Otherwise, it is \textbf{implicit}. In terms of quadratures,  if $w_{n,0}=w_{n,n}=0$ for all $n\in\mathbb{N}$, we say that the numerical integration is \textbf{open}. Otherwise,  it is \textbf{close}.
\end{defi}
\begin{rmk}\label{rmk...1}
 The series of weights $\left\{w_{n,i}\right\}_{i=0}^{n}$  varies for each $n\in\mathbb{N}$.
\end{rmk}
\noindent   In Section \ref{Subsection...LinearMulti-stepMethod} and Section \ref{Subsection..NumericalQuadrature},  we  summarize the two main ingredients of our methods designed for ODEs
with memory. One is the classical LMMs, and we remark that throughout the whole note, \emph{the term LMMs refers to  LMMs   in Definition \ref{Definition....NumericalScheme}, otherwise we would say classical LMMs to distinguish out}.  The other is the numerical quadratures,  methods involved to  approximate the integration term
$\int_0^t\vg(\vx(s),s,t)\diff s$.
\subsection{Classical Linear Multistep Methods}\label{Subsection...LinearMulti-stepMethod}
For the initial value problem 
\begin{equation}\label{eq...dx/dt=f}
    \dot{\vx}(t)=\vl(\vx(t),t),\quad {\vx}(0)=\vx_0, 
\end{equation}  
where $\vx:\overline{\mathbb{R}}_+\rightarrow \mathbb{R}^d$ and $\vl:\mathbb{R}^d\times\overline{\mathbb{R}}_+\rightarrow\mathbb{R}^d$, we consider the general finite difference scheme
\begin{equation}\label{eq...preliminary...classical}
    \vx_{n+q}=\sum_{i=1}^{q}\alpha_{q-i}\vx_{n+q-i}+h\sum_{i=0}^{q}\beta_{q-i}\vl_{n+q-i},
\end{equation}
where  $\vl_n:=\vl(\vx_n,t_n)$, and \eqref{eq...preliminary...classical}
includes all considered methods as special cases.  
Specifically, some commonly used classical LMMs designed for \eqref{eq...dx/dt=f}
 are listed out as follows
\begin{equation}
\begin{aligned}
    \text{BDF1/AM1/Backward Euler method:}
    &\quad \vx_{n}=\vx_{n-1}+h\vl_{n},\\
    \text{BDF2:}
    &\quad \vx_{n}=\frac{4}{3}\vx_{n-1}-\frac{1}{3}\vx_{n-2}+\frac{2}{3}h\vl_{n},\\
    \text{AM2/Trapezoidal rule method:}
    &\quad \vx_n=\vx_{n-1}+h\left(\frac{1}{2}\vl_n+\frac{1}{2}\vl_{n-1}\right),\\
    \text{AB1/Euler method:}
    &\quad \vx_n=\vx_{n-1}+h\vl_{n-1},\\
    \text{AB2:}
    &\quad \vx_n=\vx_{n-1}+h\left(\frac{3}{2}\vl_{n-1}-\frac{1}{2}\vl_{n-2}\right),\\
    \text{MS1/Mid-point method:}
    &\quad \vx_n=\vx_{n-2}+2h\vl_{n-1},\\
    \text{MS2/Milne method:}
    &\quad \vx_n=\vx_{n-2}+h\left(\frac{1}{3}\vl_n+\frac{4}{3}\vl_{n-1}+\frac{1}{3}\vl_{n-2}\right),
\end{aligned}
\end{equation}
and many other classical LMMs can be found in the literature of~\cite{Griffiths2010NumericalODE,Hairer1987solving,Wanner1996solving}. 
\subsection{Numerical Quadratures}\label{Subsection..NumericalQuadrature}
In order to   approximate   the integration term
\begin{equation}\label{eq...preliminary...integral}
  \int_0^t\vg(\vx(s),s,t)\diff s,\end{equation}
only  the composite rules of Newton--Cotes formulas are considered in this note, the formulas are termed \textbf{closed} when the end points of the  integration interval   are included in the formula. Otherwise, if excluded, we have an \textbf{open} Newton-Cotes quadrature. Some commonly used Newton--Cotes formulas designed for \eqref{eq...preliminary...integral} are listed out below: 
\begin{equation}
    \begin{aligned}
        \text{Trapezoidal rule (closed):}
         & \quad h\sum_{i=0}^{n-1}\frac{1}{2}(\vg_{n,i}+\vg_{n,i+1}),                      \\
        \text{Simpson's rule (closed):}
         & \quad h\sum_{i=0}^{ n/2-1}\frac{1}{3}(\vg_{n,2i}+4\vg_{n,2i+1}+\vg_{n,2i+2}),   \\
        \text{Mid-point rule (open):}
         & \quad h\sum_{i=0}^{n/2-1}2\vg_{n,2i+1},                                         \\
        \text{Trapezoidal rule (open):}
         & \quad h\sum_{i=0}^{n/3-1}\frac{3}{2}(\vg_{n,3i+1}+\vg_{n,3i+2}),                \\
        \text{Milne's rule (open):}
         & \quad h\sum_{i=0}^{n/4-1}\frac{4}{3}(2\vg_{n,4i+1}-\vg_{n,4i+2}+2\vg_{n,4i+3}),
    \end{aligned}
\end{equation}
and many other quadratures can be found in the literature of~\cite{Burden1985numerical,Davis2007methodsofnumericalintegration,Stoer2013introductiontoNumericalAnalysis}.

We  remark that the implementation of these rules on the initial problem~\eqref{eq...introduction...TargetEquation} could be a little tricky.  For instance, the closed Simpson's rule presents a   challenge in that its formula requires  $n$ to be an even number.  However, such condition may not always be met, as the time step $n$ tends to increase and varies with time due to the presence of memory effects. In the circumstances where $n$ is odd, however, we may circumvent this complication  by using Simpson's rule to approximate the term $\int_{t_1}^{t_n}\vg(\vx(s),s,t_n)\diff s$ instead of $\int_{t_0}^{t_n}\vg(\vx(s),s,t_n)\diff s$. As for the residue term $\int_{t_0}^{t_1}\vg(\vx(s),s,t_n)\diff s $, it can be   estimated by applying Simpson's rule to the time region $[t_0,t_1]$, i.e.,  $\int_{t_0}^{t_1}\vg(\vx(s),s,t_n)\diff s \approx\frac{h}{6}(\vg_{n,0}+4\vg_{n,\frac{1}{2}}+\vg_{n,1})$. Therefore, we suggest that the grid points need be computed with comparatively small time step at initial stage. 

\subsection{Truncation Errors, Consistency, Convergence and Zero-Stability}
Before we proceed to the definition of local and global truncation errors for classical LMM, we associate with \eqref{eq...preliminary...qstep}   the numerical solution at $t_n$ denoted by  $\vL_n\left({\{\vx(t_j)\}}_{j=0}^{k}\right)$, given exact solution $\{\vx(t_j)\}_{j=0}^k$ to the initial value problem \eqref{eq...introduction...TargetEquation}, i.e., $\left\{\vL_n\left({\{\vx(t_j)\}}_{j=0}^{k}\right)\right\}_{n\geq k+1}$ satisfies
\begin{equation}\label{eq...text...subsec...zerostable}
\begin{aligned}
\vL_n\left({\{\vx(t_j)\}}_{j=0}^{k}\right)&=\sum_{i=1}^{q}\alpha_{q-i} \vL_{n-i}\left({\{\vx(t_j)\}}_{j=0}^{k}\right)\\
&~~+h\sum_{i=0}^{q}\beta_{q-i}\vf\left(\vL_{n-i}\left({\{\vx(t_j)\}}_{j=0}^{k}\right),t_n\right)\\ &~~+h\sum_{i=0}^{n}w_{n,i}\vg\left( \vL_{i}\left({\{\vx(t_j)\}}_{j=0}^{k}\right),t_i,t_n\right),
\end{aligned}
\end{equation}
and for any $i\in[0:k]$, we set $\vL_i\left({\{\vx(t_j)\}}_{j=0}^{k}\right):=\vx(t_i)$ for convenience.
\begin{defi}\label{Definition....TruncationError}
    The \textbf{local truncation error} reads
    \begin{align}
        \vtau_n
        : & = \vx(t_{n})-\vL_n\left({\{\vx(t_i)\}}_{i=0}^{n-1}\right), \label{eqGroup....DefinitionofLocalTruncationError}
    \end{align}
and     we say   the numerical scheme is \textbf{consistent}  if $\Norm{\vtau_n}_{\infty}=O(h)$ for $n = 1,2,\cdots$. Moreover, a numerical scheme is said to be consistent of  \textbf{order $p$} if for  sufficiently regular differential equations \eqref{eq...introduction...TargetEquation}, $\Norm{\vtau_n}_{\infty}=O(h^{p+1})$.
\end{defi}
\begin{defi}\label{Definition....GlobalTruncationError}
    The \textbf{global truncation error} reads
    \begin{equation}\label{eq...DefinitionofGlobalTruncationError}
        \ve_n:=\vx(t_n)-\vL_n\left({\{\vx(t_0)\}}\right)=\vx(t_n)-\vL_n\left({\{\vx_0\}}\right),
    \end{equation}
    and we say the numerical method  \textbf{converges} if $\Norm{\ve_n}_{\infty}=o(1)$ holds uniformly for all $n\in\mathbb{N}$.
\end{defi}
\begin{defi}[Zero-Stability]\label{Definition...Rootcondition}
A classical LMM \eqref{eq...preliminary...classical} is called \textbf{zero-stable}, if the generating polynomial 
\begin{equation}\label{eq...def...Stable+Root}
p(s) :=  s^q -\sum_{i=1}^{q}\alpha_{q-i}s^{q-i},
\end{equation}
satisfies the \textbf{root condition}, i.e.,
    \begin{enumerate}[i)]
        \item The roots of $p(s)$ lie on or within the unit circle;
        \item The roots on the unit circle are simple.
    \end{enumerate}
\end{defi}
\noindent Another equivalent definition of zero-stability is given out as follows
\begin{defi}\label{Definition..NumeicalStability}
A classical $q$-step LMM \eqref{eq...preliminary...classical} is   \textbf{zero-stable}, if for any time  $T>0$, there exists $h_0>0$, such that for  any  given $0<h\leq h_0$, the following holds
    \begin{equation}\label{eq...DefinitionofNumericalStability}
    \Norm{{\vx}_n-\tilde{\vx}_n}_{\infty}\leq C(T)\sup_{0\leq i\leq q-1}\Norm{\vx_i-\tilde{\vx}_i}_{\infty}, 
    \end{equation}
 for  all $n\in\left[\lfloor \frac{T}{h} \rfloor\right],$
    where $\vx_n$ and $\tilde{\vx}_n$ are numerical solutions implemented with respect to different initial values $\vx_i$ and $\tilde{\vx}_i$, $i = 0,1,\ldots,q-1$, i.e., 
\[
\vx_n:=\vL_n\left({\{\vx_0,\ldots,\vx_{q-1}\}}\right), \quad    \tilde{\vx}_n:=\vL_n\left({\{\tilde{\vx}_0,\ldots,\tilde{\vx}_{q-1}\}}\right),
\]
with $\vg\equiv\vzero$ in \eqref{eq...text...subsec...zerostable}.
\end{defi}

\noindent The consistency and zero-stability are in fact the sufficient and necessary conditions for the convergence of classical LMMs, which can be represented by the Dahlquist equivalence theorem~\cite[Theorem 4.5]{Hairer1987solving}.
\begin{thm}[Convergence for classical LMMs]
\label{thm...DahCelebratedThm}
    A  classical LMM is convergent if and only if it is  consistent and zero-stable.
\end{thm}
\subsection{A-Stability} 
Another crucial concept of stability in numerical analysis is the \textbf{A-stability}~\cite[Definition 6.3]{Griffiths2010NumericalODE}. For the test problem
    \begin{equation} \label{eq...text...TestEqforLinearODE}
        \dot{x}  (t) = \lambda x(t), \quad x(0)=x_0,
    \end{equation} 
where  $x:\overline{\mathbb{R}}_+\rightarrow \mathbb{R}$, and $\lambda\in \mathbb{C}$ with negative real part, i.e., $\Re(\lambda) < 0$,  its   solution reads  \[x(t)=x_0\exp({\lambda t}).\]  Hence, $x(t) \rightarrow 0$ as $t \rightarrow \infty$, regardless of $x_0$. The A-stable methods is a class of classical LMMs, when applied to \eqref{eq...text...TestEqforLinearODE},  for any  given step size $h>0$, the numerical  solutions $\{x_n\}_{n\geq 0}$ also tend to  zero as $n \to \infty$,    regardless of the choice of starting values $x_0$, and we formalize our aspirations by the following definition~\cite[Definition 6.3]{Griffiths2010NumericalODE}. 
\begin{defi}[Absolute Stability]\label{def...AbsoluteStability}
A classical LMM is said to be \textbf{absolutely stable} if, when applied to the test problem
 \eqref{eq...text...TestEqforLinearODE}
with some given value of $\hat{h} := h \lambda$, its solutions tend to  zero as $n \to \infty$ for any choice of starting values.
\end{defi}
\noindent  We introduce the single parameter $\hat{h}$ since the parameters $h$ and $\lambda$ occur altogether as a  product, and  a classical LMM is not always   absolutely stable for every choice of $\hat{h}$, hence we
are led to the definition of the \textbf{region of absolute stability}~\cite[Definition 6.5]{Griffiths2010NumericalODE}.
\begin{defi}[Region of Absolute Stability]\label{def...RegionofAbsoluteStability}
    The set of values $\mathcal{R}$ in the complex $\hat{h}$-plane for which a classical LMM is absolutely stable forms its \textbf{region of absolute stability}.
\end{defi}
\noindent With these two definitions, we are able to give out an equivalent definition of A-stability for the classical LMMs.
\begin{defi}[A-stability]\label{def...AStability...equivalent}
   A classical LMM designed for  the initial value problem \eqref{eq...text...TestEqforLinearODE}  is said to be \textbf{A-stable} if its region of absolute stability $\mathcal{R}$ includes the entire left half plane, i.e.,
    \begin{equation}\label{eq...def....A-Stable}
        \left\{\hat{h}\in\mathbb{C}:\Re\hat{h}<0\right\}\subseteq \mathcal{R}.
    \end{equation}
\end{defi}

\subsection{Formulation as  One-step Methods}
We are now at the point where it is useful to rewrite a LMM as a one-step method in a higher dimensional space~\cite{Hairer1987solving}. In order for this,  we define $\vpsi$ associated with the LMM
\begin{equation}\label{eq...text...Zero-Stable...q-step}
  \vx_{n+q}= \sum_{i=1}^{q}\alpha_{q-i}\vx_{n+q-i}+h\sum_{i=0}^{q}\beta_{q-i}\vf_{n+q-i}+h\sum_{i=0}^{n+q}w_{n+q,i}\vg_{n+q,i},   
\end{equation}
as  follows
\begin{align*}
\vpsi&:=\beta_q\vf\left(\sum_{i=1}^{q}\alpha_{q-i}\vx_{n+q-i}+h\vpsi, t_{n+q}\right)+\sum_{i=1}^{q}\beta_{q-i}\vf(\vx_{n+q-i}, t_{n+q-i})\\
&~~+w_{n+q,n+q} \vg\left(\sum_{i=1}^{q}\alpha_{q-i}\vx_{n+q-i}+h\vpsi, t_{n+q}, t_{n+q}\right)+\sum_{i=0}^{n+q-1}w_{n+q,i} \vg(\vx_{i}, t_i, t_{n+q}),
\end{align*}
then the LMM \eqref{eq...text...Zero-Stable...q-step} can be written as 
\begin{equation}\label{eq...text...ZeroStability...MultistepasCell}
\vx_{n+q}=\sum_{i=1}^{q}\alpha_{q-i}\vx_{n+q-i}+h\vpsi_{n+q}.
\end{equation}
For any $i\in[n]$,  we introduce a $dq$-dimensional vector 
\[
\vX_i:=\left(\vx_{i+q-1}^\T, \vx_{i+q-2}^\T, \cdots, \vx_{i}^\T\right)^\T,
\]
equipped with
\[
\mA:=\begin{pmatrix}
{\alpha}_{q-1}&{\alpha}_{q-2}&\cdots & \cdot &{\alpha}_{0}\\
1 & 0 & \cdots & \cdot & 0 \\
& 1 & & . & 0 \\
& & \ddots & \vdots & \vdots \\
& & & 1 & 0\end{pmatrix}\in\sR^{q\times q},
\]
and 
\[
\ve_1:=\left(\begin{array}{c}
1 \\
0 \\
0 \\
\vdots \\
0
\end{array}\right)\in\sR^{q},
\]
then  the LMM \eqref{eq...text...Zero-Stable...q-step} can be written   in  a closed form as follows: For any $i\in[n]$,
\begin{equation}\label{eq...formulationonestep}
    \vX_{i+1}=(\mA \otimes \mI_d) \vX_i+h \mPhi_i,  
\end{equation}
where
\[
\mPhi_i:=\left(\ve_1 \otimes \mI_d\right) \vpsi_{i+q},
\]
and $\otimes$ denotes the Kronecker tensor product. 
\section{Zero-Stability of Linear Multistep Methods}

\subsection{Zero-Stability and Root Condition}
Our results concerning zero-stability are essentially based on the following lemma.
\begin{lem}[Induction lemma for zero-stability]
\label{lem..ZeroStability}
For $\lambda\in\mathbb{C}$ and $\mu>0$,  let $\{y_n\}_{n\geq 0}\subset\sR$ be a  non-negative sequence 
 satisfying
\begin{equation}\label{eq...lem...Stability...Inudction}
        \Abs{1-\lambda h}y_n=y_{n-1}+h^2\mu\sum_{i=0}^{n-1} y_{i},
\end{equation}
then given any $0<h<\frac{1}{2\Abs{\lambda}}$, the following holds for all $k\in \mathbb{N}$,
\begin{equation}\label{eq..lem...StabilityUniformBoundConditional}
        y_k\leq \exp(2\Abs{\lambda}kh+\mu k^2 h^2)y_0.
\end{equation}
\end{lem}
\begin{proof}
We prove this by induction. Suppose that \eqref{eq..lem...StabilityUniformBoundConditional} holds for $k\in[0:n-1]$, then
\begin{align*}
        y_n
        &\leq \frac{1}{\Abs{1-\lambda h}}\exp(2\Abs{\lambda}(n-1)h+\mu(n-1)^2h^2)y_0+\frac{\mu h^2}{\Abs{1-\lambda h}}\sum_{k=0}^{n-1}\exp(2\Abs{\lambda}kh+\mu k^2h^2)y_0\\
        &\leq \frac{1+\mu nh^2}{\Abs{1-\lambda h}}\exp(2\Abs{\lambda}(n-1)h+\mu(n-1)^2h^2)y_0.
\end{align*}
It suffices to show that
\begin{equation}\label{eq1...lem...stability...proof}
        \frac{1+\mu nh^2}{\Abs{1-\lambda h}}\leq\exp(2|\lambda|h+\mu(2n-1)h^2). 
\end{equation}
As we notice that  $\Abs{\lambda}h\leq \frac{1}{2}$, then 
\begin{align*}
        \frac{1}{\Abs{1-\lambda h}}
        &\leq \frac{1}{1-\Abs{\lambda}h}\leq \exp(2\Abs{\lambda}h),\\
        1+\mu nh^2
        &\leq \exp(\mu n h^2)\leq \exp(\mu(2n-1)h^2), 
\end{align*}
which finishes the proof.
\end{proof}
\noindent We  impose some   technical conditions  on $\vf$ and $\vg$ in the initial value problem \eqref{eq...introduction...TargetEquation}.
\begin{assump}[Uniform Lipschitz condition]\label{assump..Lipschitz}
There exists $L>0$, such that 
\begin{equation}
    \begin{aligned}
    \Norm{{\vf(\vx,t)-\vf(\tilde{\vx},t)}}_{\infty}&\leq L\Norm{\vx-\tilde{\vx}}_{\infty}, \\
    \Norm{\vg(\vx,s,t)-\vg(\tilde{\vx},s,t)}_{\infty}&\leq L\Norm{\vx-\tilde{\vx}}_{\infty},
    \end{aligned}
    \end{equation}
hold uniformly for all $\vx, \tilde{\vx}\in\mathbb{R}^d$ and $t\geq s>0$.
\end{assump}
\begin{thm}[Zero-stablity for LMM with memory]\label{thm...zerostable}
If a linear $q$-step method
\begin{equation}\label{eq..thm...qstep}
  \vx_{n+q}= \sum_{i=1}^{q}\alpha_{q-i}\vx_{n+q-i}+h\sum_{i=0}^{q}\beta_{q-i}\vf_{n+q-i}+h\sum_{i=0}^{n+q}w_{n+q,i}\vg_{n+q,i}, 
\end{equation}
whose generating polynomial 
\begin{equation}\label{eq...thm...Stable+Root}
p(s)  =  s^q -\sum_{i=1}^{q}\alpha_{q-i}s^{q-i},
\end{equation}
satisfies the {root condition},
then it is zero-stable. 
\end{thm}
\begin{proof}
Let $\{\vX_n\}_{n\geq 0}$ and $\{{\tilde{\vX}}_n\}_{n\geq 0}$ be the numerical solutions with  initial values $\vX_0$ and ${\tilde{\vX}}_0$, then we have 
\begin{align*}
 {\vX}_{i+1}&=(\mA \otimes \mI_d) \vX_i+h \mPhi\left(\vX_i, h\right), \\ 
 {\tilde{\vX}}_{i+1}&=(\mA \otimes \mI_d) {\tilde{\vX}}_i+h \mPhi\left({\tilde{\vX}}_i, h\right). 
\end{align*}
Taking their difference leads to

\[ {\vX}_{i+1}- {\tilde{\vX}}_{i+1} =(\mA \otimes \mI_d) \left(\vX_i-{\tilde{\vX}}_i\right)+h \left(\mPhi\left(\vX_i, h\right)-\mPhi\left({\tilde{\vX}}_i, h\right)\right),\]
and by taking $2$-norm on both sides, 
\begin{equation}\label{eqproof...2Norm}
\begin{aligned} 
\Norm{{\vX}_{i+1}- {\tilde{\vX}}_{i+1}}_2&\leq \Norm{\mA \otimes \mI_d}_{2\to 2}\Norm{{\vX}_{i}- {\tilde{\vX}}_{i}}_2+h\Norm{\vpsi\left({{\vX}}_i, h\right)-\vpsi\left({\tilde{\vX}}_i, h\right)}_2\\
&\leq \Norm{{\vX}_{i}- {\tilde{\vX}}_{i}}_2+h\Norm{\vpsi\left({{\vX}}_i, h\right)-\vpsi\left({\tilde{\vX}}_i, h\right)}_2.
\end{aligned}
\end{equation}
Remark that $\Norm{\mA \otimes \mI_d}_{2\to 2} \leq 1$ in \eqref{eqproof...2Norm} as   \eqref{eq...thm...Stable+Root} satisfies the root condition.

Note that the coefficients $\left\{\alpha_k\right\}_{k=0}^{q-1}$,  $\left\{\beta_k\right\}_{k=0}^{q}$ and  series of weights $\left\{w_{n+q,i}\right\}_{i=0}^{n+q}$ satisfy that: For each $k$ and $i$,  
\[
\Abs{\alpha_k}, \quad  \Abs{\beta_k}\leq M, \quad\Abs{w_{n+q,i}}\leq M h,
\]
for some  universal constant $M$.   By Assumption \ref{assump..Lipschitz}, we have that 
\begin{align*}
&\Norm{\vpsi\left({{\vX}}_{n}, h\right)-\vpsi\left({\tilde{\vX}}_{n}, h\right)}_2\\
&\leq  L \Abs{\beta_q}\left(\sum_{i=1}^{q}\Abs{\alpha_{q-i}}\Norm{\vx_{n+q-i}-\tilde{\vx}_{n+q-i}}_2+h\Norm{\vpsi\left({{\vX}}_{n}, h\right)-\vpsi\left({\tilde{\vX}}_{n}, h\right)}_2\right)\\
&~~+L\left(\sum_{i=1}^{q}\Abs{\beta_{q-i}}\Norm{\vx_{n+q-i}-\tilde{\vx}_{n+q-i}}_2\right)\\
&~~+L\Abs{w_{n+q,n+q}} \left(\sum_{i=1}^{q}\Abs{\alpha_{q-i}}\Norm{\vx_{n+q-i}-\tilde{\vx}_{n+q-i}}_2+h\Norm{\vpsi\left({{\vX}}_{n}, h\right)-\vpsi\left({\tilde{\vX}}_{n}, h\right)}_2\right)\\
&~~+L\sum_{i=0}^{n+q-1}\Abs{w_{n+q,i}} \Norm{\vx_{i}-\tilde{\vx}_{i}}_2, 
\end{align*}
then we obtain that 
\begin{align*}
&\Norm{\vpsi\left({{\vX}}_{n}, h\right)-\vpsi\left({\tilde{\vX}}_{n}, h\right)}_2\\    
&\leq  LM^2\sqrt{q}\Norm{{\vX}_{n}- {\tilde{\vX}}_{n}}_2+LMh\Norm{\vpsi\left({{\vX}}_{n}, h\right)-\vpsi\left({\tilde{\vX}}_{n}, h\right)}_2\\
&~~+LM\sqrt{q}\Norm{{\vX}_{n}- {\tilde{\vX}}_{n}}_2\\
&~~+LMh\left(M\sqrt{q}\Norm{{\vX}_{n}- {\tilde{\vX}}_{n}}_2+ h\Norm{\vpsi\left({{\vX}}_{n}, h\right)-\vpsi\left({\tilde{\vX}}_{n}, h\right)}_2\right)\\
 &~~+LMh\sqrt{q}\left(\sum_{j=0}^{n}\Norm{{\vX}_{j}- {\tilde{\vX}}_{j}}_2\right),
\end{align*}
hence 
\begin{equation}\label{eq...lipschitzphi}
\begin{aligned}
&\left(1-LMh-LMh^2\right)\Norm{\vpsi\left({{\vX}}_{n}, h\right)-\vpsi\left({\tilde{\vX}}_{n}, h\right)}_2\\    
&\leq  \sqrt{q}\left(LM+LM^2+LM^2h\right)\Norm{{\vX}_{n}- {\tilde{\vX}}_{n}}_2  
  +LMh\sqrt{q}\left(\sum_{j=0}^{n}\Norm{{\vX}_{j}- {\tilde{\vX}}_{j}}_2\right).
\end{aligned}
\end{equation}
Then, if we  apply the inequality above to \eqref{eqproof...2Norm}, the following holds 
\begin{equation}\label{eqproof...2Norm...sup}
\begin{aligned} 
&\Norm{{\vX}_{i+1}- {\tilde{\vX}}_{i+1}}_2\\
& \leq
\Norm{{\vX}_{i}- {\tilde{\vX}}_{i}}_2+h\Norm{\vpsi\left({{\vX}}_i, h\right)-\vpsi\left({\tilde{\vX}}_i, h\right)}_2\\
&\leq \Norm{{\vX}_{i}- {\tilde{\vX}}_{i}}_2+h\frac{\sqrt{q}\left(LM+LM^2+LM^2h\right)}{1-LMh-LMh^2}\Norm{{\vX}_{i}- {\tilde{\vX}}_{i}}_2\\
&~~+h\frac{LMh\sqrt{q}}{1-LMh-LMh^2}\left(\sum_{i=0}^{n}\Norm{{\vX}_{j}- {\tilde{\vX}}_{j}}_2\right).
\end{aligned}
\end{equation}
As we set $y_k:=\Norm{{\vX}_{k}- {\tilde{\vX}}_{k}}_2$, then for sufficiently small $h>0$, by direction application of Lemma \ref{lem..ZeroStability}, we obtain that 
\[
y_k\leq \exp(2\Abs{\lambda}kh+\mu k^2 h^2)y_0,
\]
for some universal constants $\lambda>0$ and $\mu>0$, which finishes the proof. 
\end{proof}

\section{Convergence for Linear Multistep Methods}
\subsection{Consistency for LMM with quadrature}

\begin{lem}[Consistency for LMM with memory]\label{lem...consistent}
    If the classical LMM is consistent of order $p$, $p \geq 1$, and equipped with a quadrature of order $k$, $k \geq 1$, then this LMM is consistent and has a local truncation error  $\fO(h^{min\{p,k\} + 1})$.
\end{lem}
\begin{proof}
 Recall the scheme reads
    \begin{equation}\label{eq...quadrature}
        \vx_{n+q}^{\prime} = \sum_{i=1}^{q}\alpha_{q-i} \vx_{n+q-i} + h\sum_{i=0}^{q} \beta_{q-i}\vf_{n+q-i} + h\sum_{i=0}^{n+q} w_{n+q,i}\vg_{n+q,i}.
    \end{equation}
    Suppose that we can calculate the integral explicitly and denote the exact value of the integral by $\vG_n := \int_{t_0}^{t_{n}} \vg(\vx(s),s,t_{n}) \diff s$, then the scheme reads 
    \begin{equation}\label{eq...exactintegral}
        \vx_{n+q} = \sum_{i=1}^{q}\alpha_{q-i} \vx(t_{n+q-i}) + h\sum_{i=0}^{q} \beta_{q-i}(\vf(\vx(t_{n+q-i}),t_{n+q-i}) +\vG_{n+q-i}).
    \end{equation}
As the classical LMM is consistent of order $p$, then the local truncation error of \eqref{eq...exactintegral} satisfies 
    \begin{equation}
        \vtau^\prime_{n+q} :=  \vx(t_{n+q}) - \vx_{n+q}^{\prime}=\fO(h^{p+1}). 
    \end{equation}
   Since we have to apply quadratures $\vI_{n} :=  \sum_{i = 0}^{n} w^\prime_{n,i}\vg(\vx(t_i), t_i,t_n)$ to approximate $\vG_n$, and as the quadrature is of order $k$, we obtain that: For any $j\in[n+q]$, 
   \[\Norm{\vG_{j} - \vI_{j} }_{\infty}= \fO(h^{k}), \] 
   hence the truncation error of \eqref{eq...quadrature} reads
    \begin{equation}
    \begin{aligned}
        \vtau_{n+q } &=  \vx(t_{n+q}) -\vx_{n+q}\\
        &= \vx(t_{n+q}) -\vx^{\prime}_{n+q}+\vx^{\prime}_{n+q}- \sum_{i=1}^{q}\alpha_{q-i} \vx_{n+q-i} - h\sum_{i=0}^{q} \beta_{q-i}(\vf_{n+q-i}+\vI_{n+q-i} )\\
        &=\fO(h^{p+1})+h\fO(h^{k})=O(h^{\min\{p,k\}+1}).
    \end{aligned}
    \end{equation}

\end{proof}

\subsection{Convergence}\label{subsection...convergence}
In the case of LMM with quadrature, we still have the Dahlquist equivalence theorem. 
\begin{thm}[Convergence for LMM with memory]\label{thm...convergence}
    Under the Assumption~\ref{assump..Lipschitz}, if a linear $q$-step method
\begin{equation}\label{thm}
  \vx_{n+q}= \sum_{i=1}^{q}\alpha_{q-i}\vx_{n+q-i}+h\sum_{i=0}^{q}\beta_{q-i}\vf_{n+q-i}+h\sum_{i=0}^{n+q}w_{n+q,i}\vg_{n+q,i}, 
\end{equation}  satisfies the {root condition}, and consistent, then it is convergent. Moreover, if the method has  local truncation errors of order  $p$,  then the global error is convergent of order $p$. 
\end{thm}
\begin{proof} 
We recall the formulation \eqref{eq...formulationonestep}, where \eqref{thm} can be written  as one-step methods. Moreover, we have $\Norm{(\mA \otimes \mI_d)}_{2\to 2} \leq 1$ since \eqref{thm} satisfies the root condition. Hence, we are able to reproduce all the proof procedures in \cite[Theorem 4.5]{Hairer1987solving}.  
\end{proof}
\section{Weak A-Stability}\label{section....A-StabilityODEwithMemory}
\subsection{Definition of Weak A-Stability}\label{subsection.....EquivalentDefinitionforAstability}
Recall that the general solution for   the test problem
    \begin{equation} 
        \dot{x}  (t) = \lambda x(t), \quad x(0)=x_0,
    \end{equation} 
has the form $ x(t) = x_0\exp( \lambda t)$, and as  $\Re(\lambda)<0$, we have that $x(t) \to 0$ as $t \to\infty$, regardless of the value of $x_0$.
Hence, in order to extend the  notion of A-Stability  to   ODEs with memory,  we endeavor to choose an appropriate class of test problems, and  dig into the long time behavior of the solutions to the test problems.
Our test problem is chosen to be a linear equation equipped with a stationary memory kernel, i.e.,
\begin{equation}\label{eq..text...LinearTestEquationwithMemory}
    \dot{x}(t)=\lambda x+\int_0^t k(t-s)x(s)\diff{s}.
\end{equation}
\emph{WLOG, we assume throughout this note that}   $\lambda\in {\mathbb{R}}$,  $k(\cdot):\overline{\mathbb{R}}_+\to \mathbb{R}$, \emph{and furthermore}, $k(x)$ \emph{is an} ${L}^1$\emph{ function over the region} $\overline{\mathbb{R}}_+$, \emph{ i.e.}, $k(t)\in {L}^1\left(\overline{\mathbb{R}}_+\right)$.  In order for the property of the long time behavior of  the solution $x(\cdot)$ to our test  problem \eqref{eq..text...LinearTestEquationwithMemory}   to be discovered,   the \textbf{Laplace Transform}\cite{spiegel1965laplace} needs to be employed.
\begin{defi}[Laplace transform]
\label{def..LaplaceTransform}
    The \textbf{Laplace transform} of a locally integrable function $x(\cdot):\overline{\mathbb{R}}_+ \to \sR$  is defined by
    \begin{equation}\label{eq...def..LaplaceTransform}
        \mathcal{L}\{x\}(w)=\int_{0}^{\infty}x(t)\exp(-wt) \diff t,
    \end{equation}
    for those $w\in\mathbb{C}$  where the integral makes sense.
\end{defi}
\begin{prop}[Sufficient condition for Absolute stability]
\label{prop...SolutiontoTestEquationwithMemory}
    If
    \begin{equation}\label{eq..prop..condition}
        \lambda+\int_{0}^{\infty} \Abs{k(\xi)}  \diff \xi<0,
    \end{equation}
    then for all choices of initial condition $x_0$, the exact solution   to   \eqref{eq..text...LinearTestEquationwithMemory} satisfies
    \begin{equation}\label{eq...prop...LimitatInfinity}
        \lim_{t\rightarrow+\infty} x(t)=0.
    \end{equation}

\end{prop}
\begin{proof}
    Taking Laplace transform on \eqref{eq..text...LinearTestEquationwithMemory}, we have
    \begin{equation*}
        w\mathcal{L}\{x\}(w)-x(0)=\lambda\mathcal{L}\{x\}(w)+\mathcal{L}\{k\}(w)\mathcal{L}\{x\}(w).
    \end{equation*}
    Thus
    \begin{equation*}
        \mathcal{L}\{x\}(w)
        =\frac{x(0)}{w-\lambda-\mathcal{L}\{k\}(w)}.
    \end{equation*}
    Apply final value theorem, we obtain that
    \begin{equation*}
        \lim_{t\rightarrow+\infty} x(t)
        = \lim_{w\rightarrow 0}w \mathcal{L}\{x\}(w)
        = \lim_{w\rightarrow 0}\frac{wx(0)}{w-\lambda-\mathcal{L}\{k\}(w)}.
    \end{equation*}
    The standard assumptions for the final value theorem\cite{schiff1999laplace,spiegel1965laplace}
    ~require that the Laplace transform have all of its poles either in the open-left-half plane or at the origin, with at most a single pole at the origin. Our condition \eqref{eq..prop..condition} implies that $\lambda < 0$, and
    \begin{equation}\label{eq..proof...prop...SolutiontoTestEquationwithMemory}
        \lambda+\mathcal{L}\{k\}(0)<0.
    \end{equation}
    Then $w \mathcal{L} \{x\} (w)$ have no poles in the open-right-half plane and on the imaginary line, which finishes the proof.
 
\end{proof}
\noindent Proposition \ref{prop...SolutiontoTestEquationwithMemory} reveals that both   $\lambda$ and the kernel   $k(\cdot)$ have impact on the long time behavior of the solution $x(\cdot)$, hence it is natural that both components have impact on our notions of the regions of absolute stability  for ODEs with memory. 
\begin{defi}[Regions of Absolute Stability]\label{def..RegionsofA--Stability}
    We define the \textbf{region of absolute stability for the exact solution}  to  the test problem \eqref{eq..text...LinearTestEquationwithMemory} as 
\begin{equation}\label{eq...def...RegionStabilityExactSoln}
        \mathcal{R}_\text{exact}
        := \left\{(\lambda,k(\cdot)): \lim_{t\rightarrow+\infty}x(t)=0 \ \text{for any initial values $x_0$}\right\},
    \end{equation}
    and we define the \textbf{region of absolute stability for the numerical scheme} as
 \begin{equation}\label{eq...def...RegionStabilityNumericalSoln}
        \mathcal{R}_\text{num}
        := \left\{(\lambda,k(\cdot)): \lim_{n\rightarrow+\infty}x_n=0 \ \text{for any  step size $h>0$ and   initial values $x_0$}\right\}.
    \end{equation}
\end{defi}
\noindent Then directly from Proposition \ref{prop...SolutiontoTestEquationwithMemory}, we obtain that
\begin{equation}\label{eq...text....SetIncludueExactSoln}
    \left\{(\lambda,k(\cdot)): \lambda+\int_{0}^{\infty} \Abs{k(\xi)} \diff \xi <0\right\}\subseteq \mathcal{R}_\text{exact},
\end{equation} 
and  we are able to give out the definition of \textbf{weak A-stability}.
\begin{defi}[Weak A-Stability]\label{def..WeakAbosoluteStability}
    A LMM designed for test problem \eqref{eq..text...LinearTestEquationwithMemory}  is \textbf{weak A-stable} if the following relation holds
    \begin{equation}\label{eq...def...A-StabilitySetContaining}
        \left\{(\lambda,k(\cdot)): \lambda+\int_{0}^{\infty} \Abs{k(\xi)} \diff \xi <0\right\}\subseteq \mathcal{R}_\text{num}.
    \end{equation}
\end{defi}
\noindent   We would like to show that   Definition \ref{def..WeakAbosoluteStability} is coherent with Definition \ref{def...AbsoluteStability} in that  for the test problem \eqref{eq...text...TestEqforLinearODE}, since $\lambda\in\sR$ and $k(\cdot)\equiv 0$,  then the region of absolute stability for the exact solution   reads
$$\mathcal{R}_{\text{exact}}=\left\{(\lambda,0):  \lambda+0=\lambda <0 \right\},$$
and the  region of absolute stability for the numerical scheme  $\mathcal{R}_\text{num}$ coincides with $\mathcal{R}$ in Definition \ref{def...RegionofAbsoluteStability}, hence  the set relation \eqref{eq...def...A-StabilitySetContaining}  reads
\begin{equation*}
   \left\{(\lambda,0): \lambda  <0\right\}\subseteq \mathcal{R}_\text{num}=\mathcal{R},
\end{equation*}
which is exactly the case in \eqref{eq...def....A-Stable}.
\subsection{Weak A-Stability for One-step Methods}
Our results concerning weak A-stability are essentially based on the following lemma.
\begin{lem}[Induction lemma for weak A-stability]
\label{lem..AStability}
For $\lambda\in\sR$, given $0<\beta\leq 1$, and a  non-negative  absolutely convergent  infinite series   $\{k_n\}_{n\geq 1}$ whose summands equal to $\kappa$, i.e., $\sum_{n=1}^\infty k_n=\kappa$, then for  all $h>0$ and  ${\lambda} <-\kappa$,  consider the non-negative sequence    $\{y_n\}_{n\geq 0}$   satisfying
\begin{equation}\label{eq...lem...AStability}
(1-\beta\lambda h)y_n=[1+(1-\beta)\lambda h]y_{n-1}+\beta h\sum_{i=1}^{n-1} k_{n-i}y_{i}+(1-\beta) h\sum_{j=1}^{n-2} k_{n-j}y_{j},
\end{equation}
we have that   
\begin{equation} 
        \lim_{n\to\infty}y_n=0.
\end{equation}
\end{lem}
\begin{proof}
WLOG, we set $y_0=1$.
We claim that given any $m\in\mathbb{N}^+$, there exists   $N_m\in\mathbb{N}^+$, such that for all $n\geq N_m$,
\begin{equation}\label{eq...proof...ynandalpha}
  0<y_n\leq \alpha^m,  
\end{equation}
for some constant 
\[
\alpha:=
 \frac{1}{2}\left(\frac{1+(1-\beta)\lambda h+\kappa h}{1-\beta\lambda h}+1\right) <1.
 \]

 \noindent We  will prove \eqref{eq...proof...ynandalpha} by   induction.

\noindent(i). The base case: for $m=1$, we show that there is a constant $N_1$ such that for $n\geq N_1$, $0<y_n\leq \alpha$. Indeed, we set $N_1=1$. For $n=1$, we have 
\[y_1=\frac{1+(1-\beta)\lambda h}{1-\beta\lambda h}\leq \alpha.\]  
For $n=2$, we have 
\[y_2\leq \frac{1+(1-\beta)\lambda h+\kappa h}{1-\beta\lambda h}\leq \alpha.\]  
For $n\geq 2$, suppose that for all $j\in [n-1]$, we have $y_k\leq \alpha<1$, then for $j=n$, we have
\[
y_n\leq \frac{1+(1-\beta)\lambda h+h\sum_{i=1}^{n-1} k_{n-i}}{1-\beta\lambda h}\leq\frac{1+(1-\beta)\lambda h+\kappa h}{1-\beta\lambda h}\leq \alpha.
\]
   (ii). The inductive step: For $m\geq 2$,  suppose that  for all $l\in[m-1]$,  there exists $N_l$ such that for $n\geq N_l$, $0<y_n\leq \alpha^l$. Our goal is to prove that for $l=m$, there exists $N_m$, such that for $n\geq N_m$, $0<y_n\leq\alpha^m$.

As the  infinite series   $\{k_n\}_{n\geq 1}$ is absolute convergent, 
hence for   large enough $n$, there exists $T_m\in\mathbb{N}^+$, such that for any $\epsilon>0$,
\begin{equation*}
\sum_{i=T_m}^{n-1}k_{i}\leq \epsilon,
\end{equation*}
we choose $\epsilon=\frac{\alpha^{m-1}(-\lambda-\kappa)}{2}$, i.e., 
\begin{equation*}
\sum_{i=T_m}^{n-1}k_{i}\leq \frac{\alpha^{m-1}(-\lambda-\kappa)}{2},
\end{equation*}
and we set $N_m=N_{m-1}+T_m$, as  we have $y_i\leq \alpha^{m-1}$ for $i\geq n-T_m+1$,  then for $n\geq N_m$,
\begin{equation} 
\begin{aligned}
y_n&\leq \frac{(1+(1-\beta)\lambda h)\alpha^{m-1}+h\sum_{i=1}^{n-T_m}k_{n-i}+h\sum_{i=n-T_m+1}^{n-1}k_{n-i}\alpha^{m-1}}{1-\beta\lambda h}\\
&\leq\frac{h\sum_{i=T_m}^{n-1}k_{i}}{1-\beta\lambda h}+ \frac{\alpha^{m-1}(1+(1-\beta)\lambda h+h\sum_{i=1}^{\infty}k_{i})}{1-\beta\lambda h}\\
&\leq\frac{\alpha^{m-1}(-\lambda h-\kappa h)}{2(1-\beta\lambda h)}+\frac{\alpha^{m-1}(1+(1-\beta)\lambda h+\kappa h)}{1-\beta\lambda h}=\alpha^m,
\end{aligned}
\end{equation}
which   finishes our proof.
\end{proof}
\noindent
In the case where $\beta=1$, we have that 
\begin{prop}[Weak A-stablity for Backward Euler method]
\label{prop...BackwardEulerisAstable}
The Open Backward Euler method designed for \eqref{eq..text...LinearTestEquationwithMemory} 
\begin{equation}
 \frac{x_{n}-x_{n-1}}{h}=\lambda x_{n}+h\sum_{i=1}^{n-1}k_{n-i}x_{i},    
\end{equation}
equipped with the weight   $k_{n-i}:=\int_{ih}^{(i+1)h}k(t_n-s)\diff{s}$  is weak A-stable.
 \end{prop}
\noindent
In the case where $\beta=\frac{1}{2}$, we have that 
\begin{prop}[Weak A-stability for Trapezoidal method]
The Open Trapezoidal method  designed for \eqref{eq..text...LinearTestEquationwithMemory}
\begin{equation}
 \frac{x_{n}-x_{n-1}}{h}=\frac{\lambda}{2} \left(x_{n}+x_{n-1}\right)+\frac{h}{2} \sum_{i=1}^{n-1} k_{n-i}y_{i}+ \frac{h}{2}\sum_{j=1}^{n-2} k_{n-j}y_{j},    
\end{equation}
equipped with the weight   $k_{n-i} =\int_{ih}^{(i+1)h}k(t_n-s)\diff{s}$  is weak A-stable. 
\end{prop}
\section{Numerical experiments}
In this section, we present several numerical examples to verify the proposed zero-stability and weak A-stability definitions and theorems for ordinary differential equations (ODEs) with memory. We consider different types of memory kernels and initial conditions, and compare the numerical solutions obtained by the proposed methods with the exact solutions or reference solutions. We also measure the errors and convergence rates of the methods, and demonstrate the stability of each method. The numerical examples illustrate the applicability and validity of our theoretical results for ODEs with memory.

\subsection{Zero-stability}
Firstly we perform numerical experiments to verify the zero-stability that we have proved above. As we have Theorem~\ref{thm...zerostable} that all the classical LMMs that satisfy the root condition are zero-stable after being equipped with open quadrature. In this subsection, we consider the most commonly used methods, i.e., Forward Euler, Backward Euler and BDF2 that are all zero-stable in our conclusion. We vary the time step from $h=1/2^2$ to $h=1/2^8$, and observe the performance of the numerical solutions.

\subsubsection{Example 1}
We consider the following ODE with memory:
\begin{equation}\label{ex...1}
    \dot{x}(t) = x(t) - \int_0^t 2\text{e}^{-1.9(t-s)}x(s)\,\diff s, \quad x(0) = 1 .
\end{equation}
We can see that the memory item in this example is an exponential function, thus we can obtain the exact soluiton of the equation through Laplace Transform. We first fix the quadrature as open Mid-point rule, and compare the stability between BDF2, Backward Euler and Forward Euler methods. The results are shown in Figure~\ref{fig:ex1-1}, where we can observe that when the step size $h$ is relatively small, the three methods can approximate the true solution well. However, as we increase $h$, different methods have errors due to accuracy issues. With the same quadrature as Mid-point rule, the performance of the Forward Euler is relatively poor, while that of the Backward Euler and BDF2 are similarly better. Moreover, to study the effects of different quadratures, we test BDF2 with Trapezoidal rule, Simpson rule and Milne rule. All of them have higher order than Mid-point rule and the results are shown in Figure~\ref{fig:ex1-2}. We can see that all the methods perform better than that with Mid-point rule as we expect. From this example, we conclude that these two implicit methods have superior numerical performance than the explicit Forward Euler. As for the usage of quadratures, the above tests show that the stable domain highly depends on the quadrature rules we use, which should be carefully chosen in applications.
\begin{figure}[htbp]
    \centering
    \includegraphics[width = 0.3\linewidth]{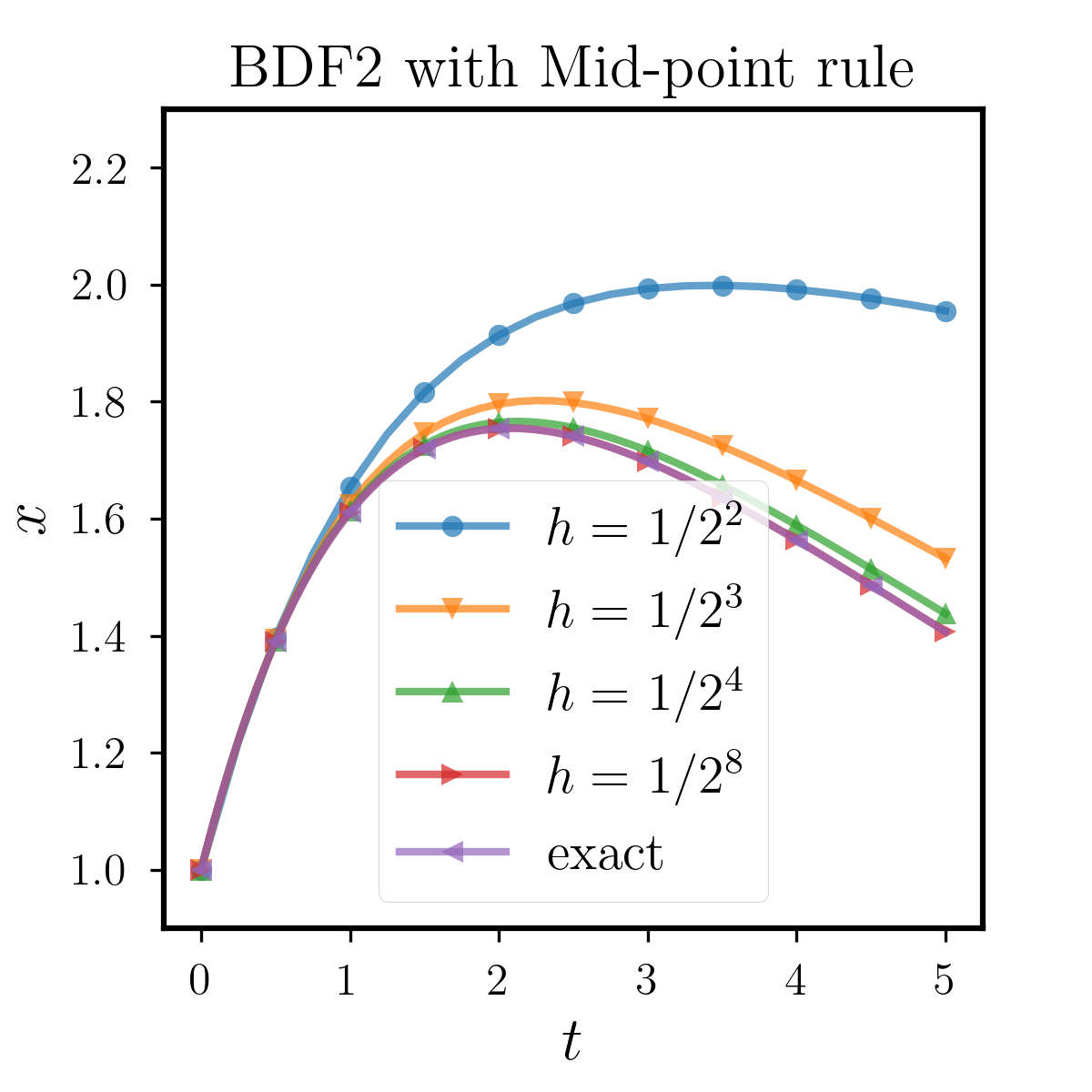}
    \includegraphics[width = 0.3\linewidth]{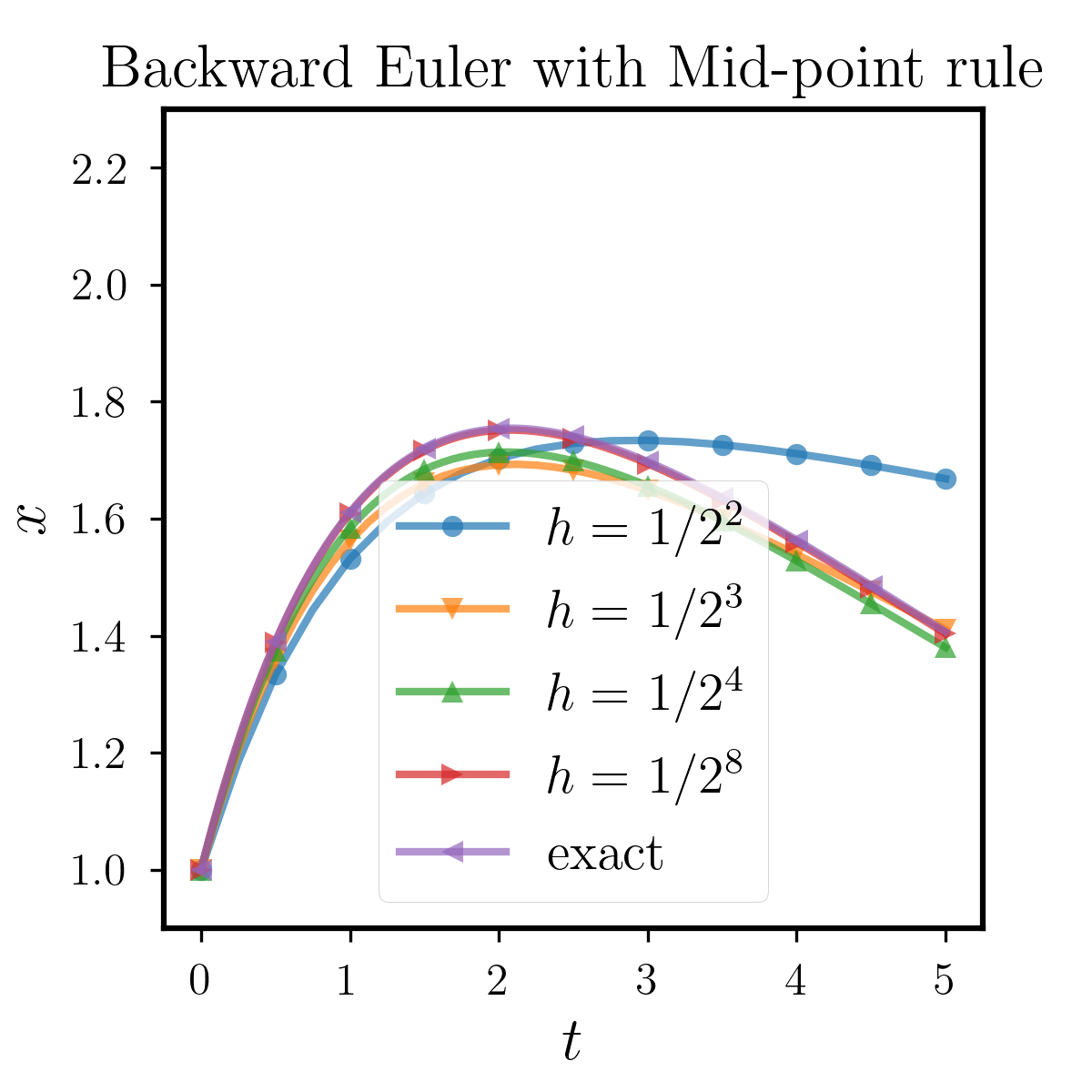}
    \includegraphics[width = 0.3\linewidth]{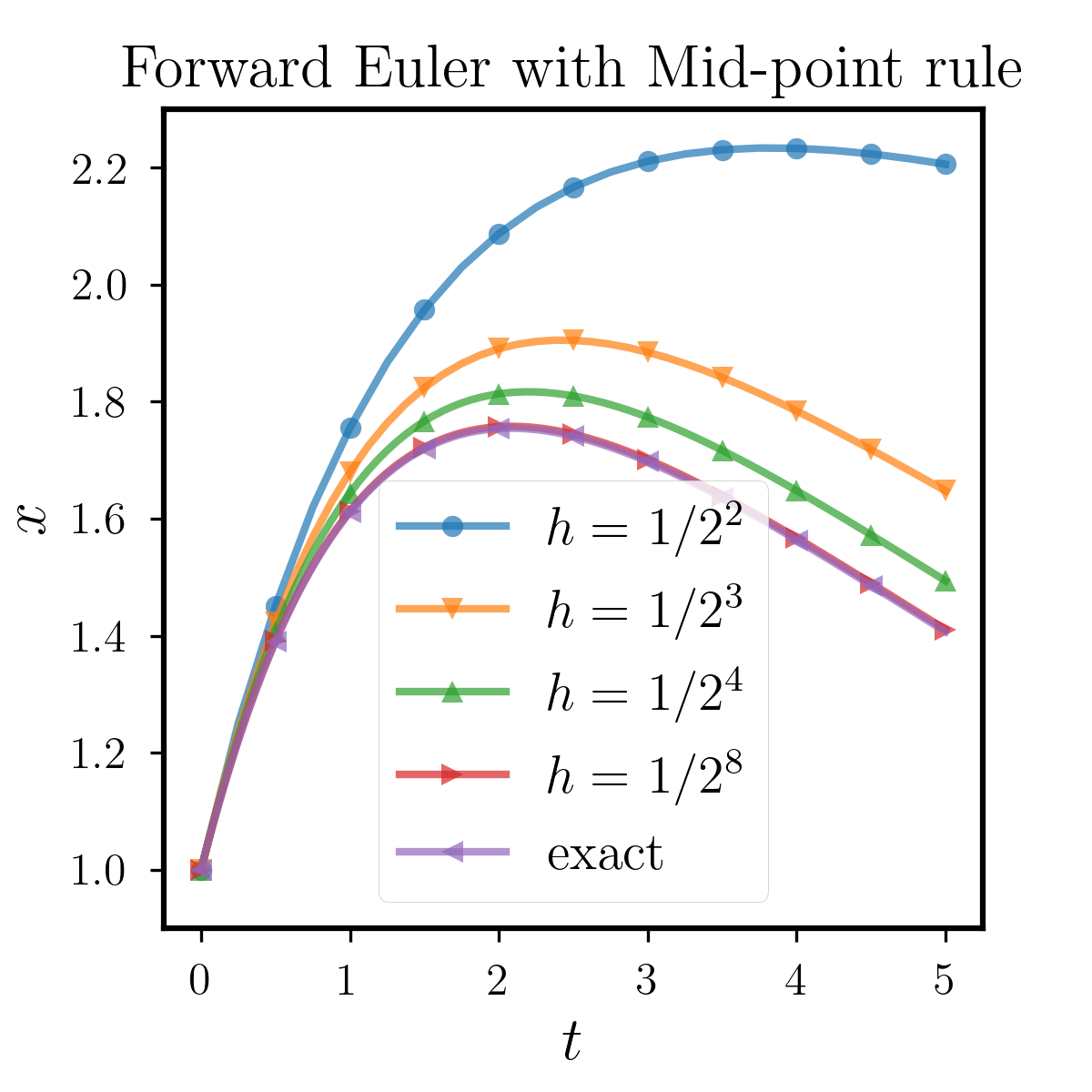}
    \caption{Performance of different methods with the same quadrature. From the left to the right, we apply BDF2, Backward Euler and Forward Euler with Mid-point rule to \eqref{ex...1}. The time steps $h = 1/2^2,1/2^3,1/2^4,1/2^8$,  the $5$ solid curves represent numerical solution of different $h$ and exact solution which is obtained by Laplace Transform.
    }
    \label{fig:ex1-1}
\end{figure}

\begin{figure}[htbp]
    \centering
    \includegraphics[width = 0.3\linewidth]{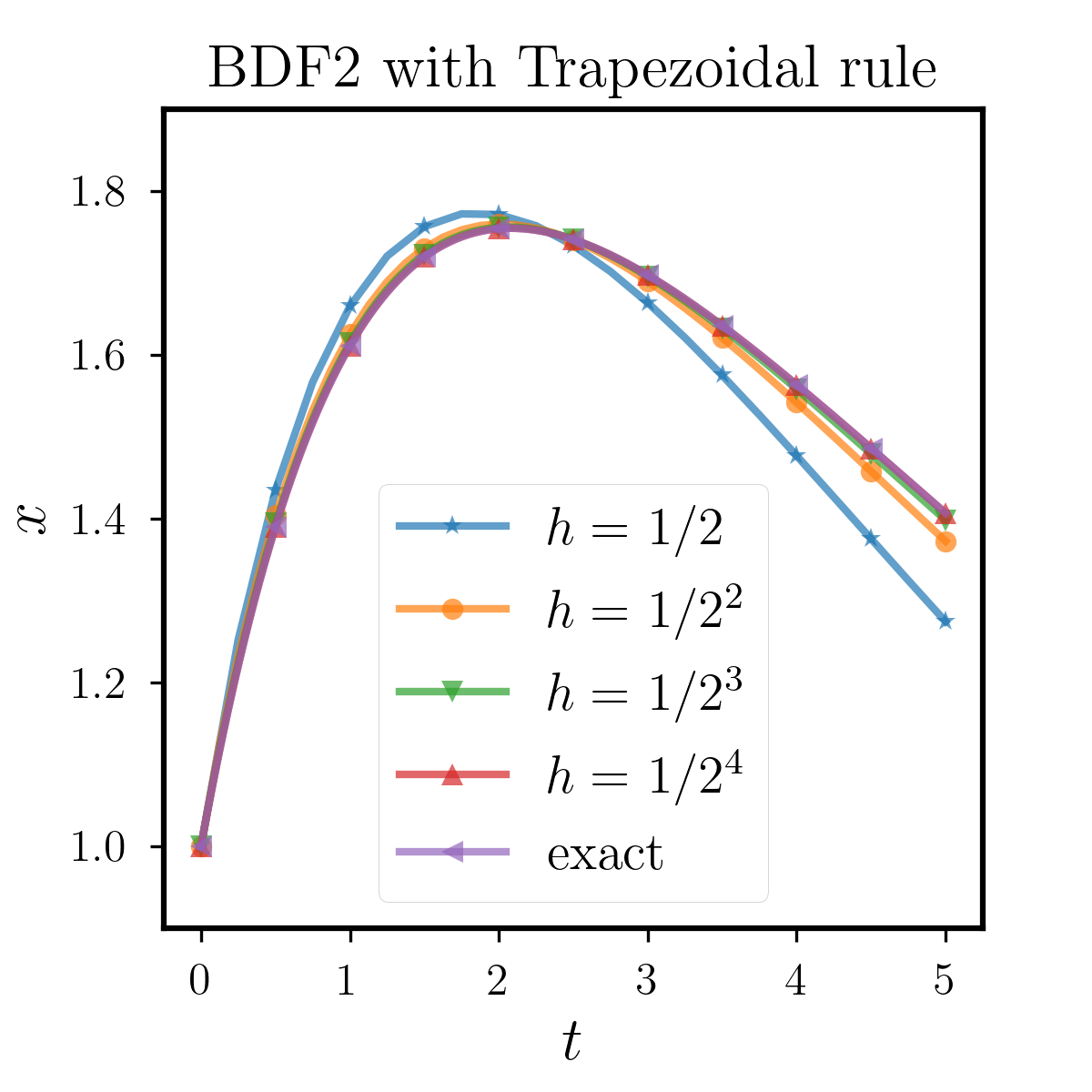}
    \includegraphics[width = 0.3\linewidth]{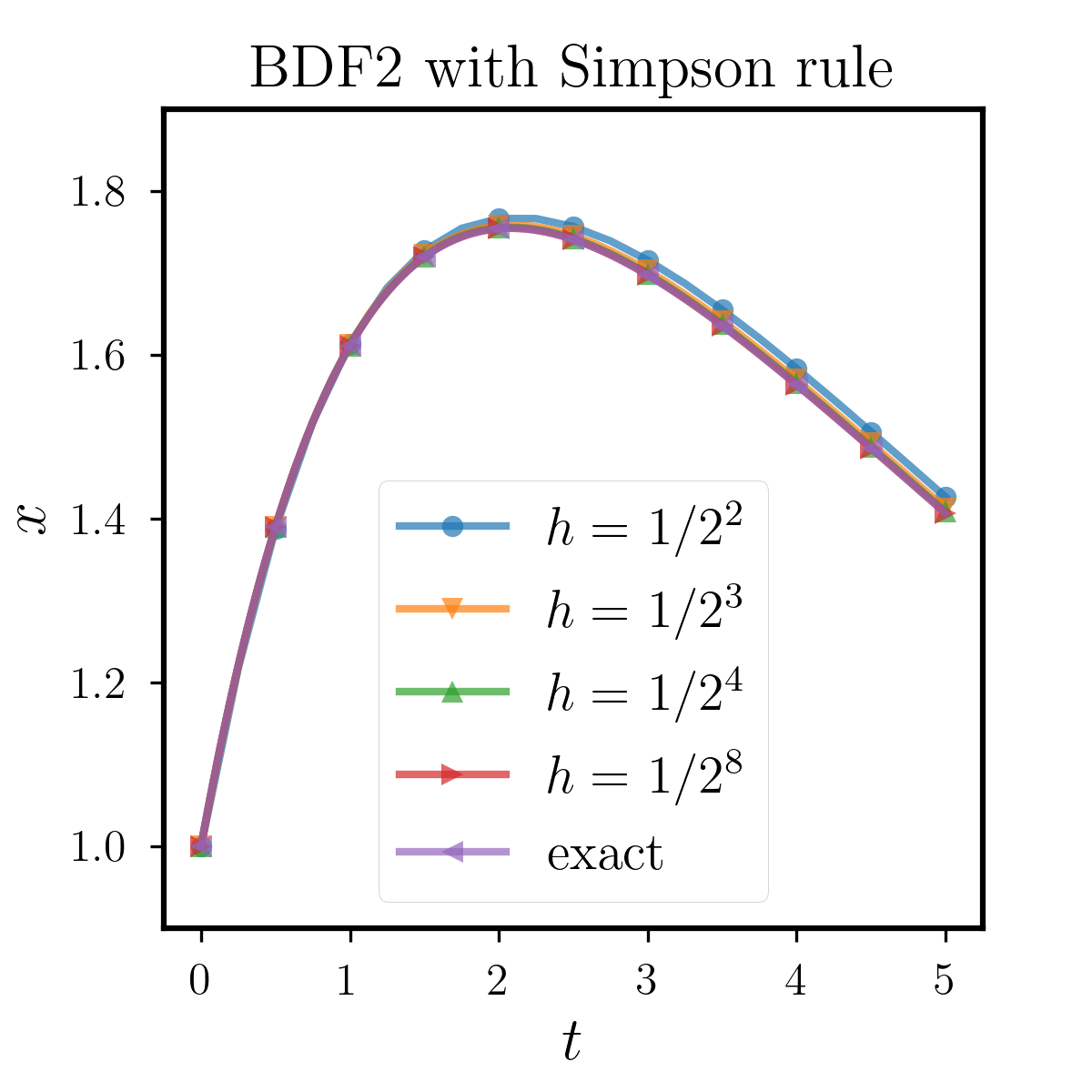}
    \includegraphics[width = 0.3\linewidth]{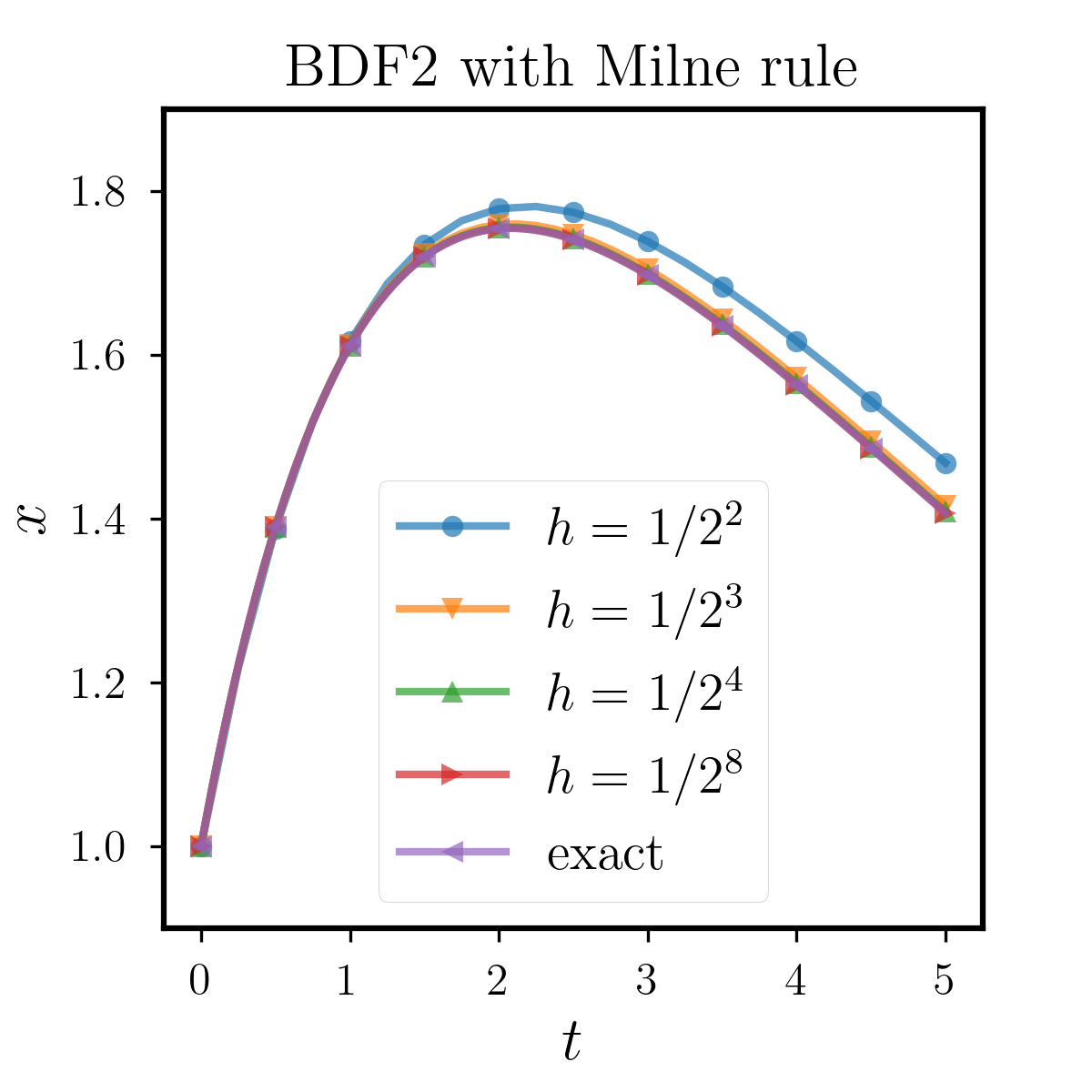}
    \caption{Performance of different quadratures of BDF2. From the left to the right, we apply the BDF2 with Trapezoidal rule, Simpson rule and Milne rule to \eqref{ex...1}. The time steps $h = 1/2^2,1/2^3,1/2^4,1/2^8$, the $5$ solid curves represent numerical solution of different $h$ and exact solution which is obtained by Laplace Transform.}
     \label{fig:ex1-2}
\end{figure}

\subsubsection{Example 2}
In the previous example, we consider an exponential memory term, which has a relatively weak impact on the solution. In this example, we change the kernel to a power function then the equation reads:
\begin{equation}\label{ex...2}
    \dot{x}(t) = x(t) - \int_{0}^{t}\frac{10}{(t-s+1)^2} \diff s, \quad x(0) = 1.
\end{equation}
At this time, we can not derive the analytic solution so that we treat the numerical solution with $h = 1/2^8$ as the exact solution. The result are shown in Figure~\ref{fig:ex2}, where we observe that the solution of the \eqref{ex...2} is periodically oscillatory and decaying. In this example, we still use the Mid-point rule to compare the differences between the three different multi-step methods. It can be seen that for various step sizes, the BDF2 can capture this periodic oscillation well, while for the two Euler methods, as the step size is relatively larger, the numerical solution has a significantly higher or lower amplitude. Therefore, BDF2 has a clear advantage over the two Euler methods in this example. 

\begin{figure}[htbp]
    \centering
    \includegraphics[width = 0.3\linewidth]{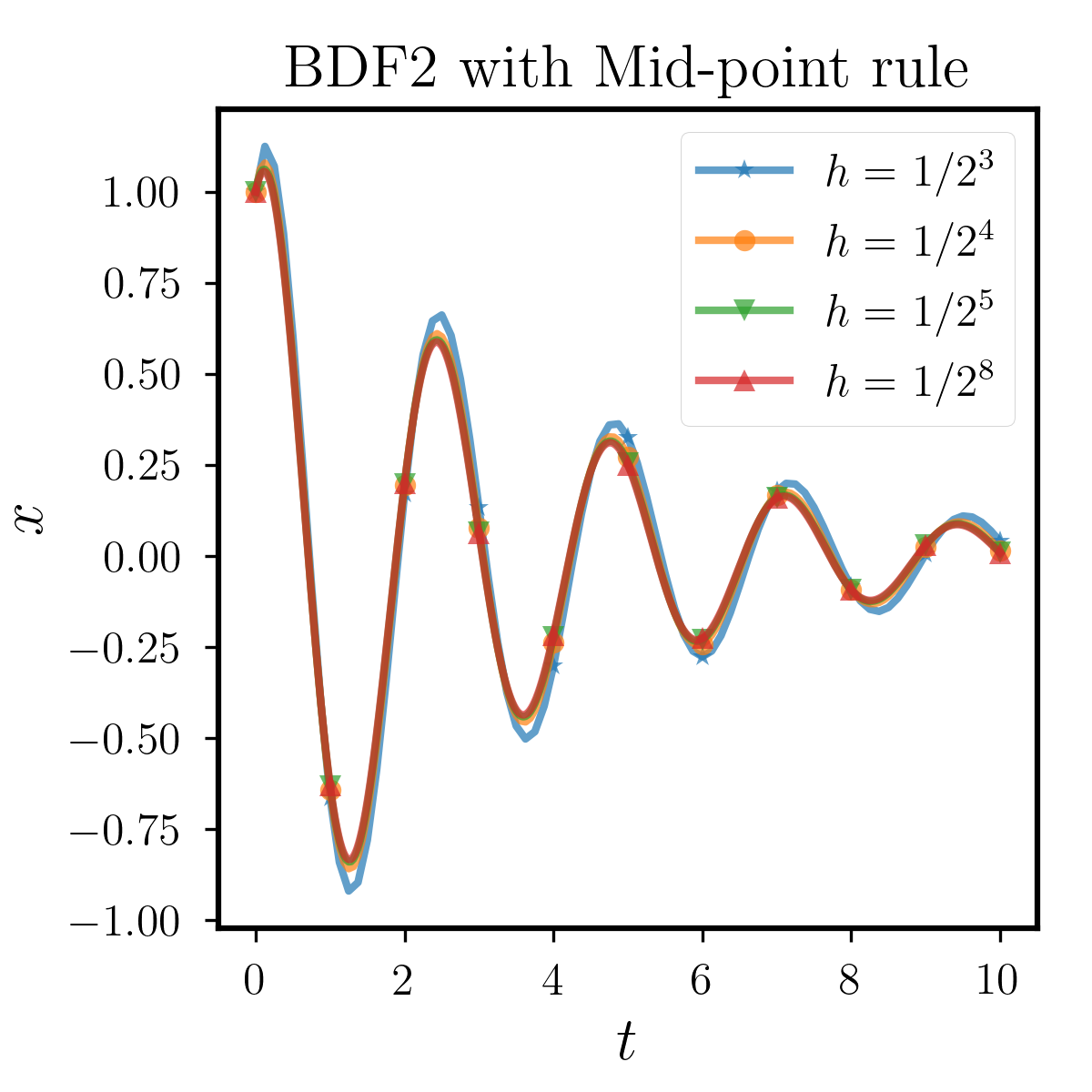}
    \includegraphics[width = 0.3\linewidth]{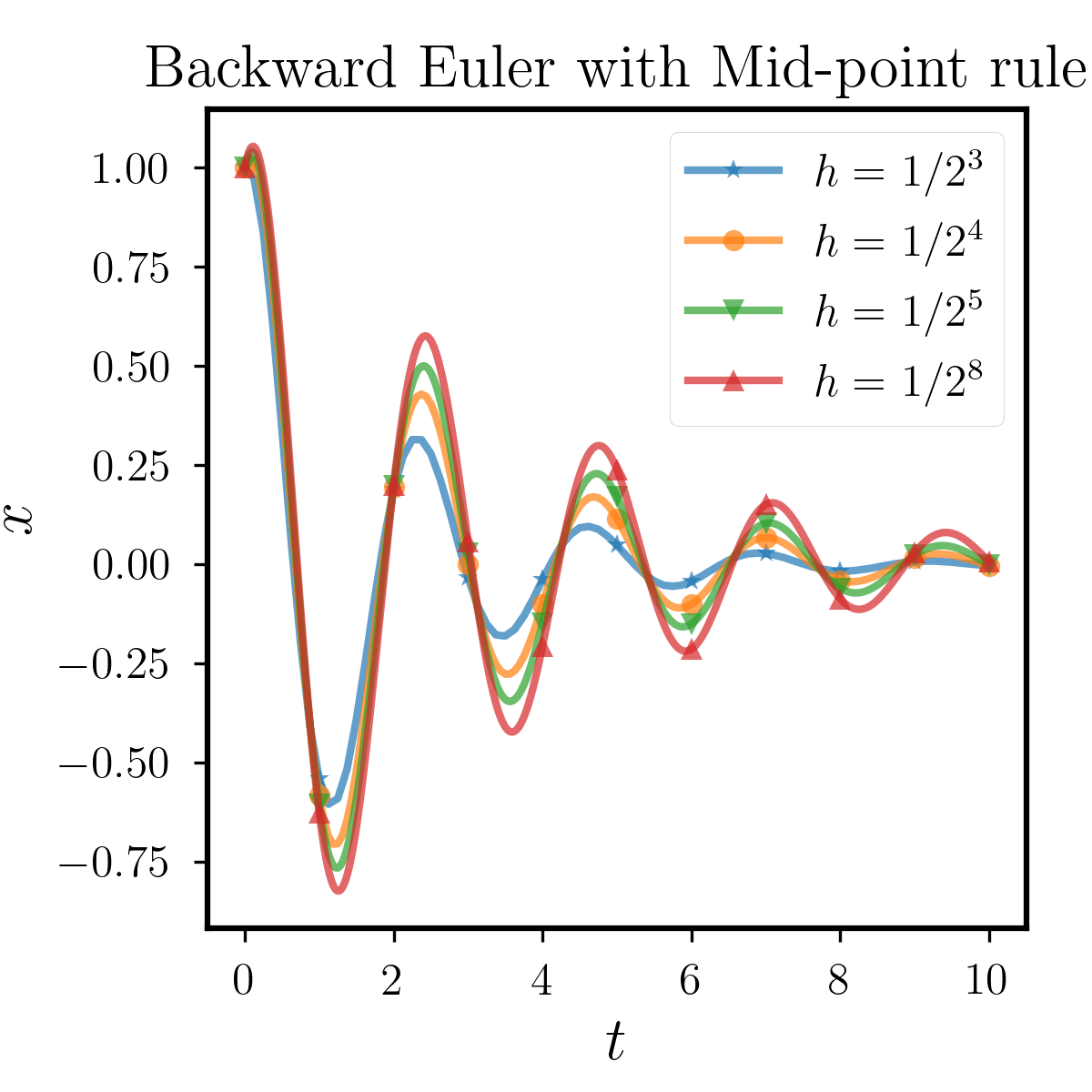}
    \includegraphics[width = 0.3\linewidth]{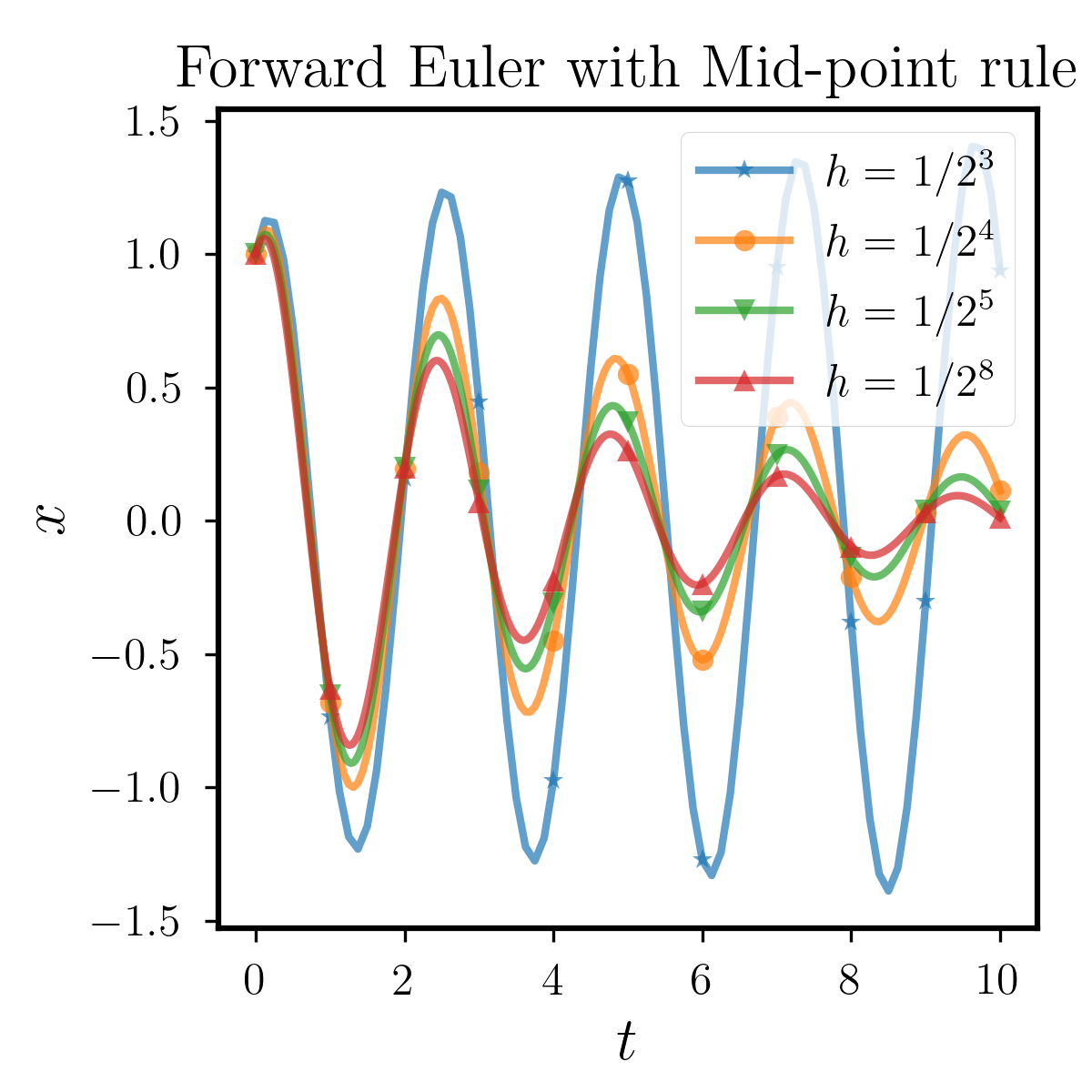}
    \caption{Performance of different methods with the same quadrature. From the left to the right, we apply BDF2, Backward Euler and Forward Euler with Mid-point rule to \eqref{ex...2}. The time steps $h = 1/2^3,1/2^4,1/2^5,1/2^8$, the $4$ solid curves represent numerical solution of different $h$.}
    \label{fig:ex2}
\end{figure}

\subsection{Convergence}
In this subsection, we test the convergence of the proposed methods for ODEs with memory. In the following examples, we set various time steps and measure the error of the numerical solutions. We use the following error measure:
\begin{equation}\label{eq...error}
    \text{error} := \norm{x(t) - x_{\text{exact}}(t)}_\infty = \max_{t_i\in[0,T]} \Abs{x(t_i) - x_{\text{exact}}(t_i)},
\end{equation}
where the exact solution $x_{\text{exact}}(t)$ is obtained analytically and $t_i$ is a grid point. As we have Lemma~\ref{lem...consistent}, the order of LMM is determined by both the classical LMM and the quadrature. Thus we will combine different method with different quadratures and verify the order of the error we define as \eqref{eq...error}.

\subsubsection{Example 3}
In this example we consider the equation
\begin{equation}\label{ex...3}
    \dot{x}(t) = x(t) - \int_0^t 2\text{e}^{-(t-s)} x(s) \diff s, \quad x(0) = 1.
\end{equation}
This equation has an analytical solution:
\begin{equation}
    x(t) = \sin(t) + \cos(t).
\end{equation}
We test MS2 with both Milne rule and Mid-point rule for convergence from $t = 0$ to $t = 10$. Recall that MS2 is a fourth-order method, the order of composed Milne rule and composed Mid-point rule is $4$ and $2$. The results are shown in Figure~\ref{fig:ex3}, where both the $x$-axis and $y$-axis are on a logarithmic scale. It shows that MS2 with Milne rule is fourth-order accurate as the order of both two components are $4$. While MS2 with Mid-point rule is second-order accurate, since it is restricted by the accuracy of quadrature.
\begin{figure}[htbp]
    \centering
    \includegraphics[width = 0.3\linewidth]{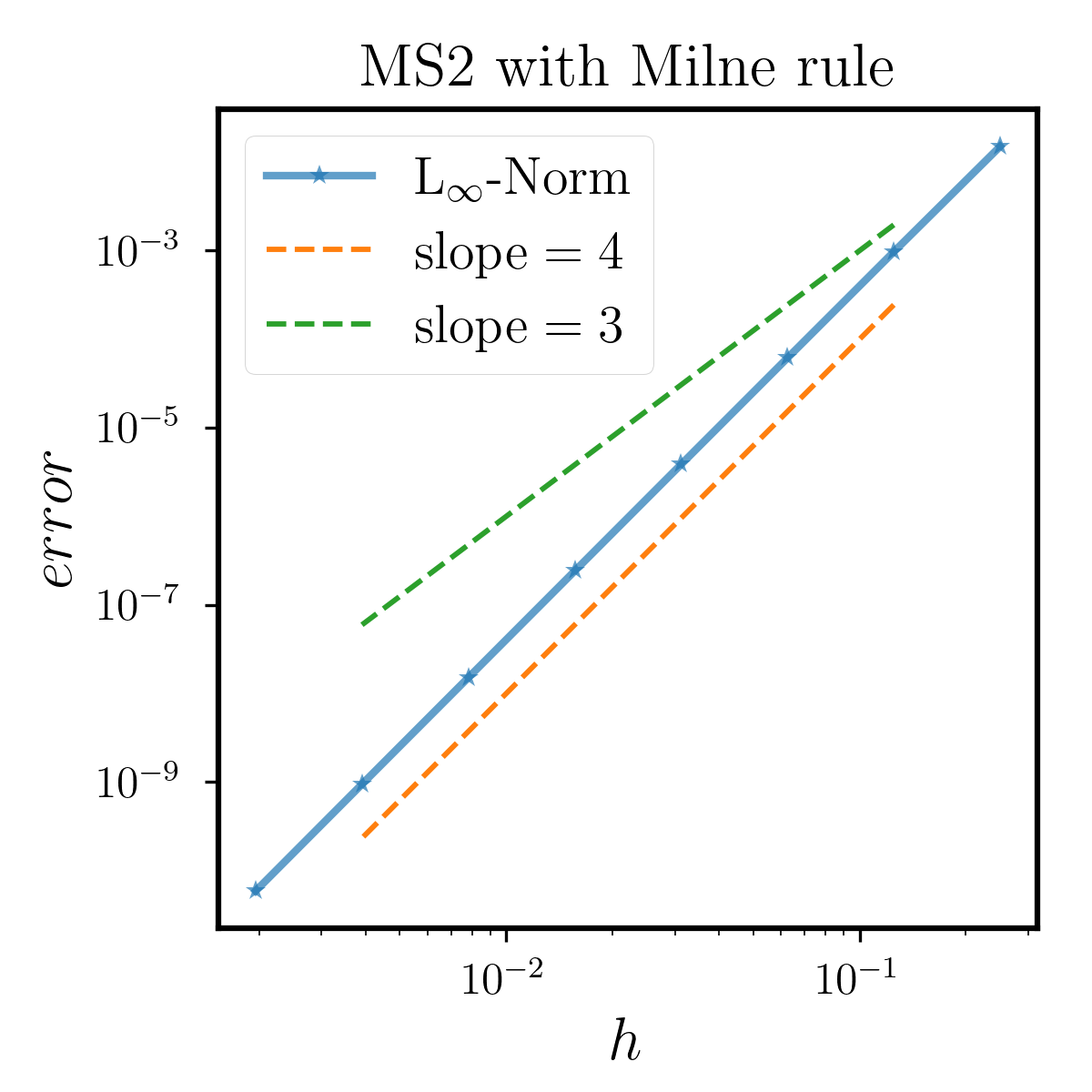}
    \includegraphics[width = 0.3\linewidth]{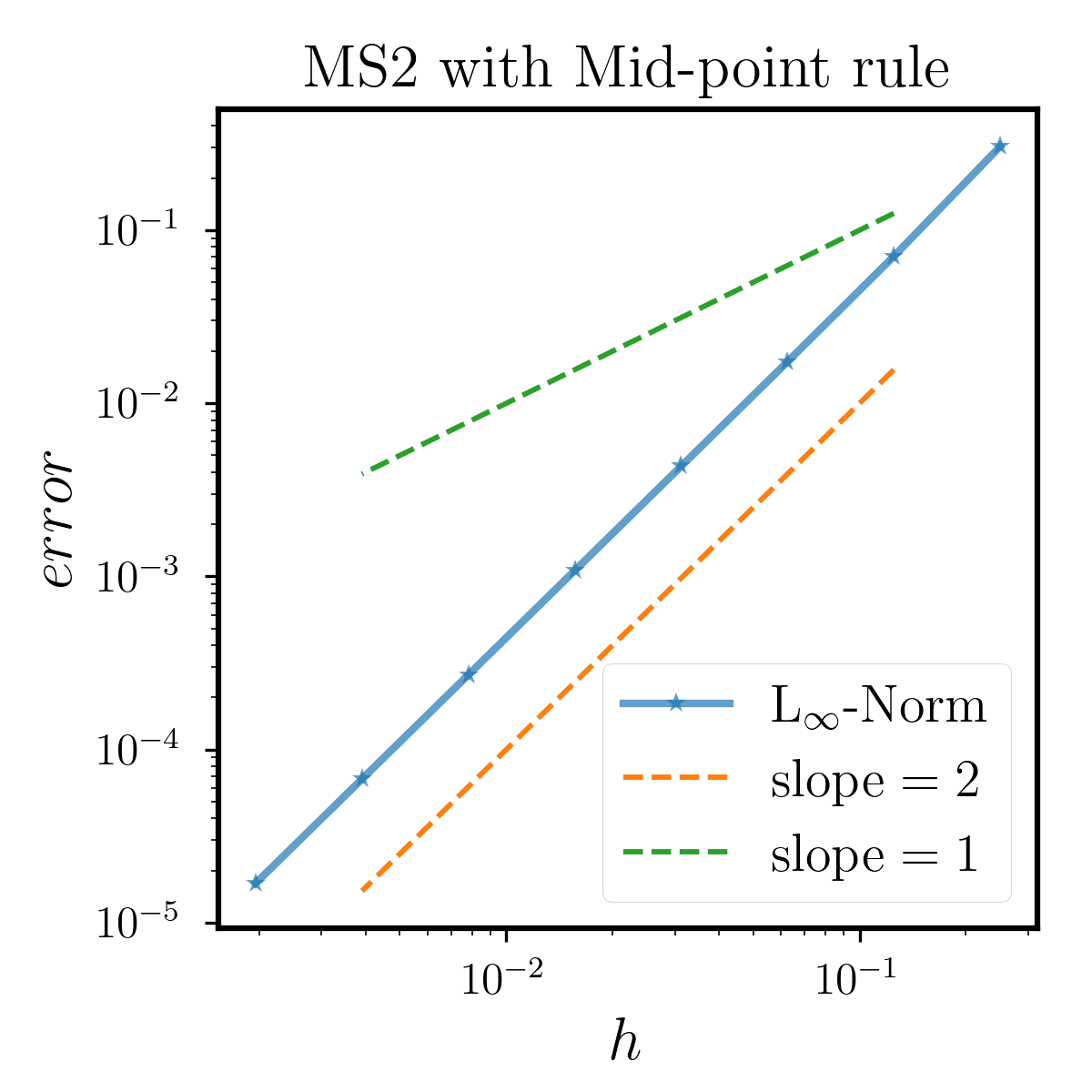}
    \caption{Order verification of MS2 performed on \eqref{ex...3} up to $t = 10$. Left side: MS2 + Milne rule, right side: MS2 + Mid-point rule. The time step sizes $h = 1/2^3, 1/2^4, 1/2^5, 1/2^6, 1/2^7, 1/2^8, 1/2^9, 1/2^{10}$. The blue solid line represents the error \eqref{eq...error} curve with different step sizes, and the dashed lines represent the comparison line of slopes. }
    \label{fig:ex3}
\end{figure}

\subsubsection{Example 4}
In this example, we change the memory kernel and use BDF2 instead to solve the equation which reads
\begin{equation}\label{ex...4}
    \dot{x}(t) = -x(t) + \int_0^t 8\text{e}^{-3(t-s)} x(s) \diff s, \quad x(0) = 1.
\end{equation}
The analytical solution is
\begin{equation}
    x(t) = \left( \frac{1}{3} \sinh (3t)  + \cosh (3t) \right) \text{e}^{-2t}.
\end{equation}
We test the convergence rate of BDF2 with both Milne rule and Mid-point rule from $t = 0$ to $t = 5$. The results are shown in Figure~\ref{fig:ex4}, which shows that both methods are second-order accurate. Unlike the former example, the two methods are mainly restricted by the accuracy of the BDF2, as it is just of order $2$.

\begin{figure}[htbp]
    \centering
    \includegraphics[width = 0.3\linewidth]{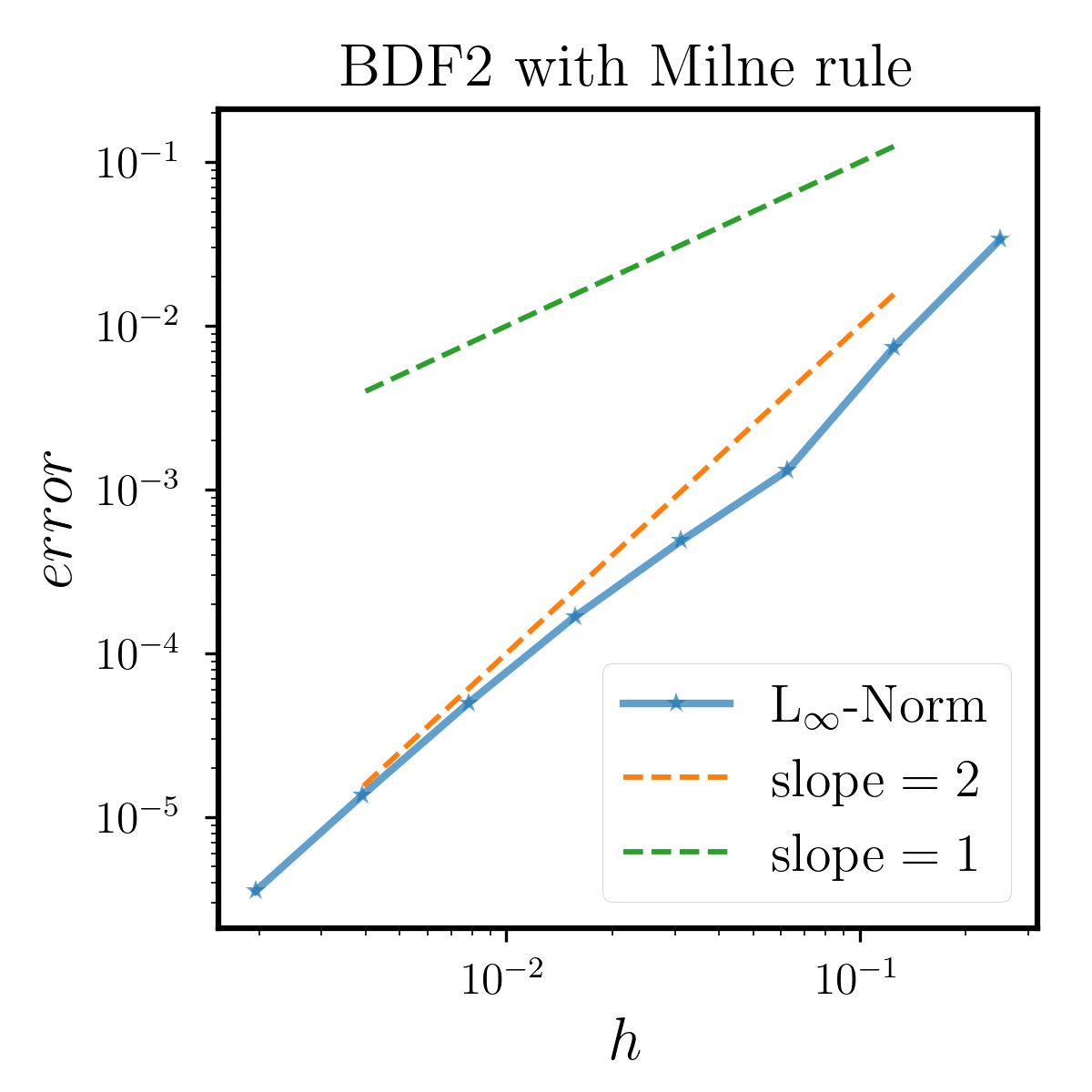}
    \includegraphics[width = 0.3\linewidth]{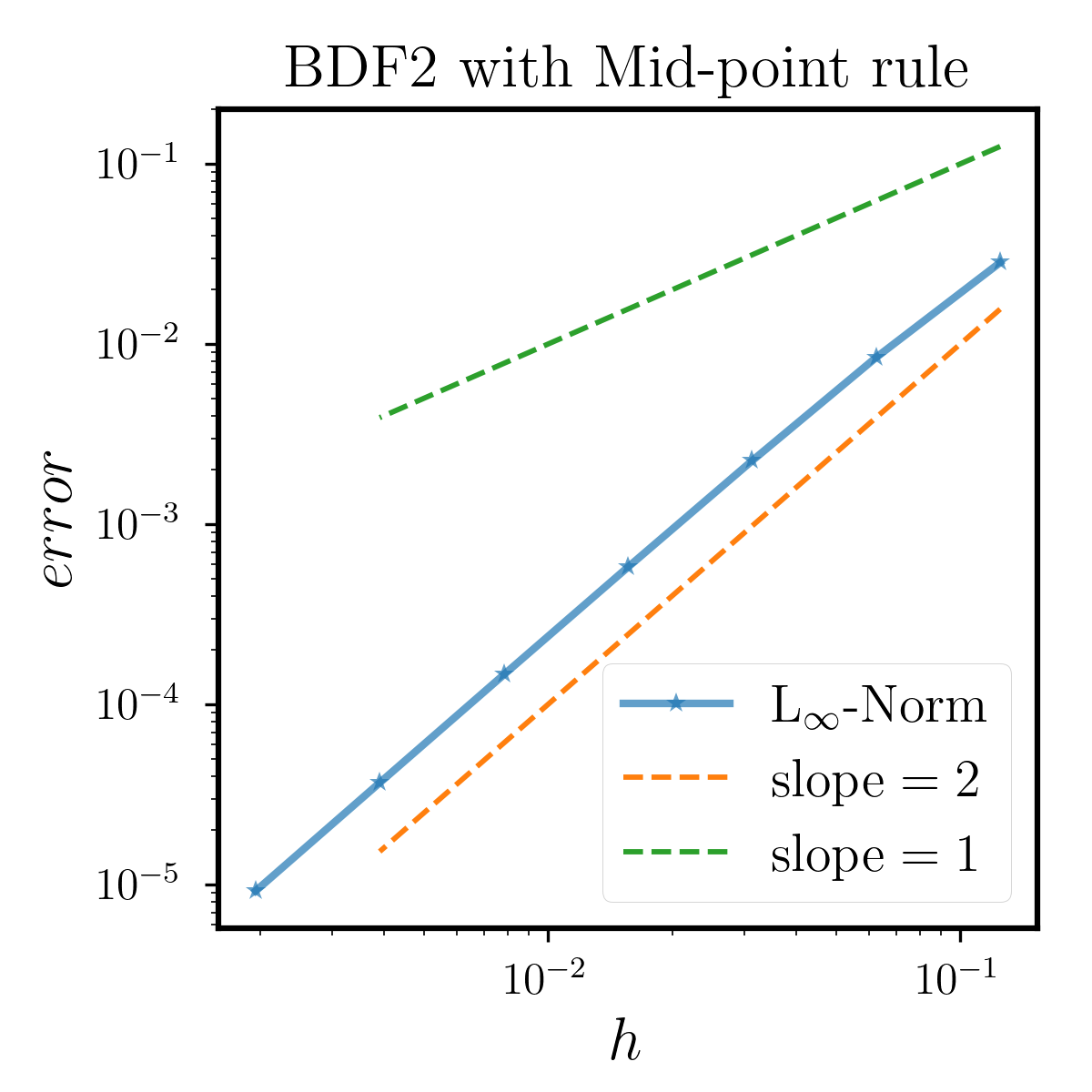}
    \caption{Order verification of BDF2 performed on \eqref{ex...4} up to $t = 5$. Left side: BDF2 with Milne rule, right side: BDF2 with Mid-point rule. The time step sizes $h = 1/2^4, 1/2^5, 1/2^6, 1/2^7, 1/2^8, 1/2^9, 1/2^{10}$. The blue solid line represents the error \eqref{eq...error} curve with different step sizes, and the dashed lines represent the comparison line of slopes. }
    \label{fig:ex4}
\end{figure}

From the above two examples, we can see that the order of the method is determined by both the classical LMM and the quadrature rule, i.e., when we use MS2 the error of integral part is dominant. From this prospective, it is unnecessary to use a method with a large difference of order between the two components. 

\subsection{Weak A-stability}
Finally, we verify the weak A-stability results proved in Section~\ref{section....A-StabilityODEwithMemory}, where  we give the modified definition~\ref{def..WeakAbosoluteStability} and obtain that Backward Euler is weak A-stable. 

\subsubsection{Example 5}
In this example, we can consider the equation 
\begin{equation}\label{ex...5}
    \dot{x} = -11 x + \int_0^t 10\text{e}^{-(t-s)} x(s) \diff s, \quad x(0) = 1 .
\end{equation}
Here $\lambda = -11$, $ \kappa = 10$, thus $\lambda + \kappa <0$ and the solution of \eqref{ex...5} is absolutely stable. Numerically, we use Mid-point as the quadrature, with $h$ values of $1/4$ and $1/8$, and employ the Forward Euler method and the Backward Euler method. The results are shown in the left side of Figure~\ref{fig:ex56}. It is observed that the numerical solutions obtained from the Backward Euler method, which is weak A-stable, are stable for both $h$. While for the Forward Euler method, the numerical solution obtained with a larger $h$, i.e., $h = 1/4$ is divergent, indicating that the method is not weak A-stable.  

\subsubsection{Example 6}
In this example, we discuss our modified definition weak A-stability, where the ``weakness" results from that we treat \eqref{eq...text....SetIncludueExactSoln} as the region of absolute stability for the exact solution which should be a proper subset of the region. However, as we can derive the analytic solution when the kernel is exponential. We find that $\left\{(\lambda,k(\cdot)): \lambda+\int_{0}^{\infty} \Abs{k(\xi)} \diff \xi <0\right\}$ indeed equals to $\mathcal{R}_\text{exact}$ in the case of exponential kernel. In this example, we consider the power function kernel and study the distinction of the two sets in \eqref{eq...text....SetIncludueExactSoln}. We look at the following equation: 
\begin{equation}\label{ex...6}
        \dot{x}(t) = \lambda x + \int_0^t \frac{10}{{(t-s+1)}^2} x(s) \diff s , \quad x(0) = 1,
\end{equation}
with $\lambda + \kappa =-10 + 10 = 0$ or $\lambda + \kappa = -10.1 + 10 = -0.1$.
\begin{figure}[htbp]
    \centering
    \includegraphics[width = 0.3\linewidth]{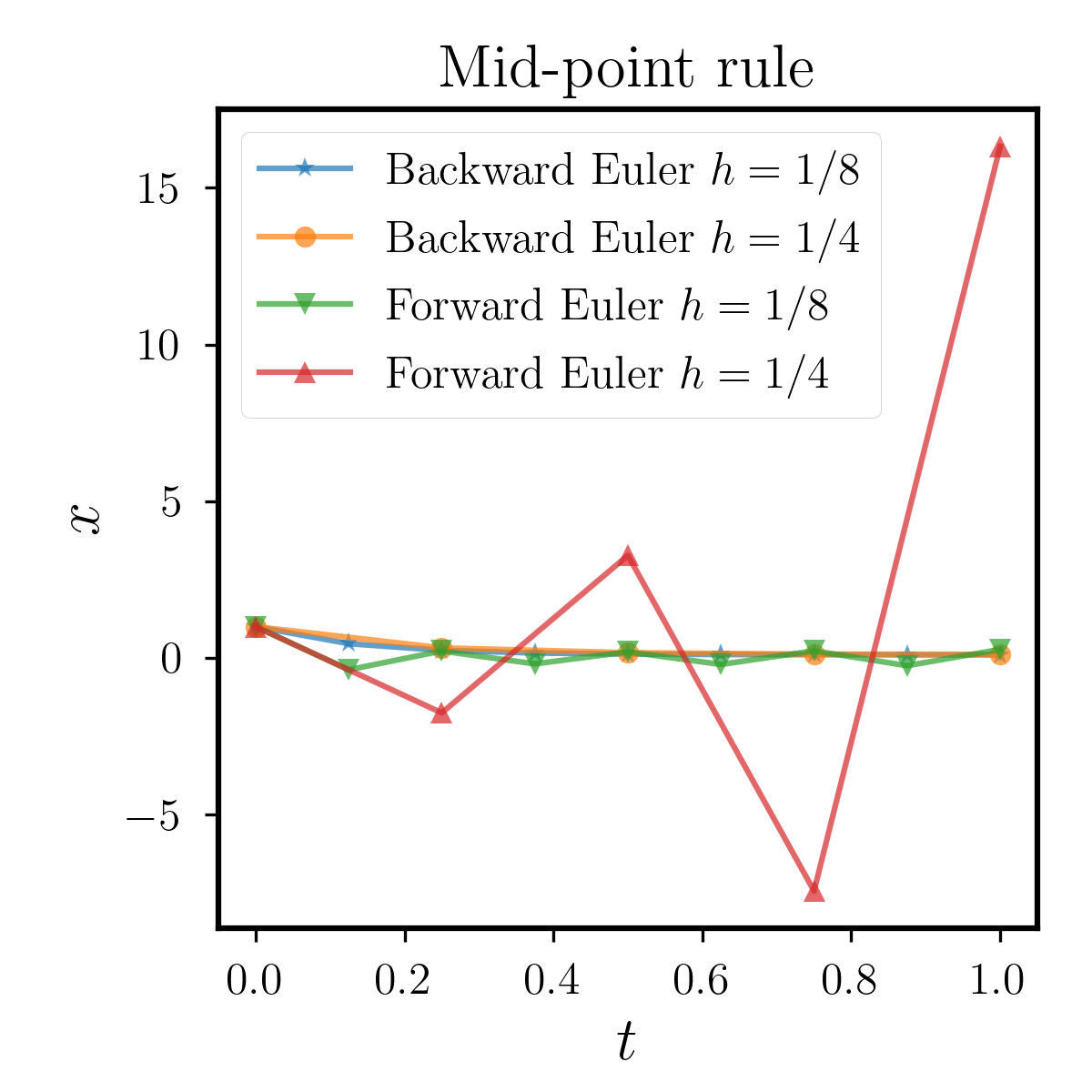}
    \includegraphics[width = 0.3\linewidth]{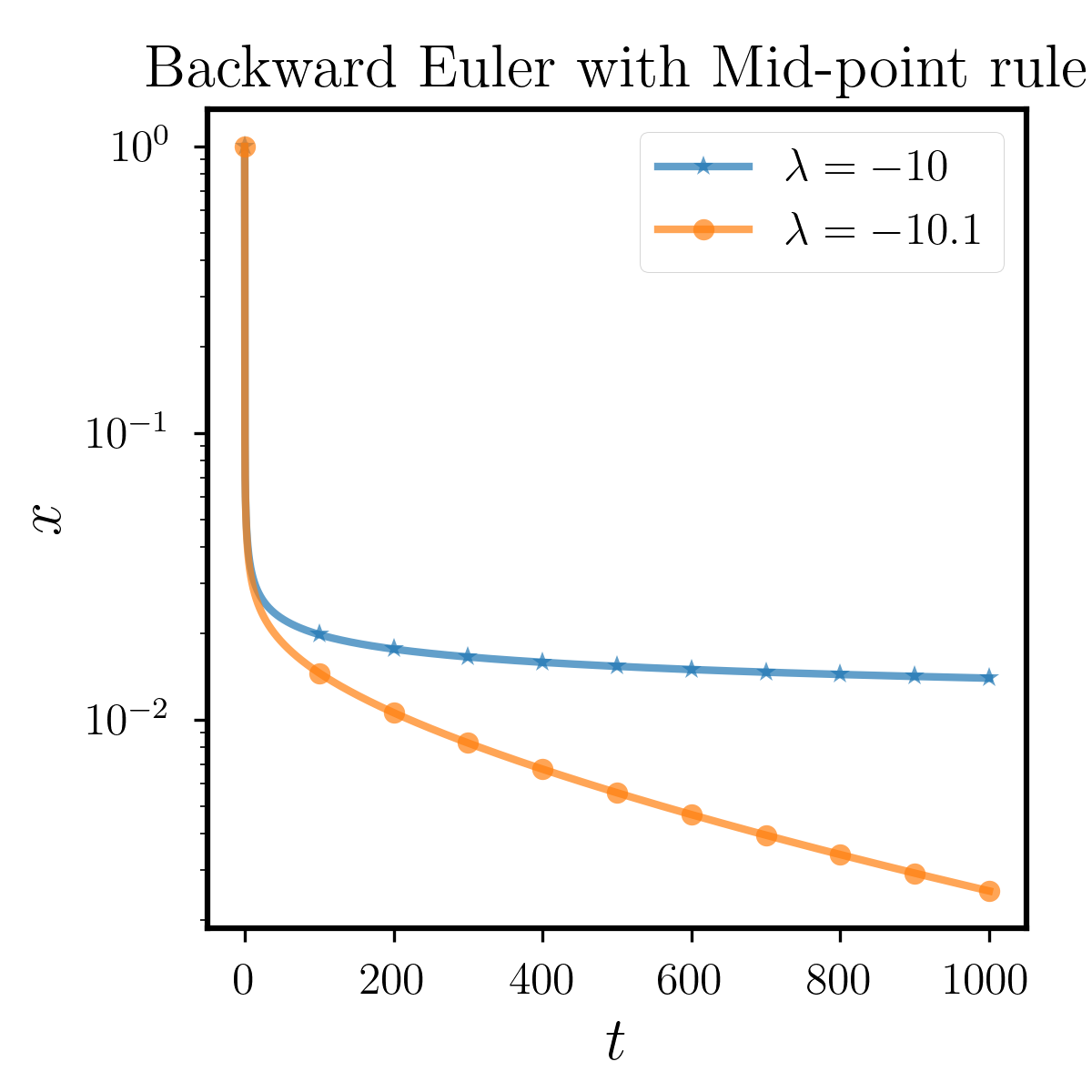}
    \caption{Left: In Example 5, we use Mid-point rule and compare the Backward Euler and Forward Euler with $h = 1/4, 1/8$, we can observe that the solution of Forward Euler is not stable when $h = 1/8$. 
    Right: In Example 6, we set the step size as $1/2^8$ and apply Backward Euler with Mid-point rule to \eqref{ex...6} with different value $\lambda$, we can find that the solution is absolutely stable when $\lambda + \kappa<0$, otherwise converges to a constant when $\lambda + \kappa = 0$. }
    \label{fig:ex56}
\end{figure}
In this experiment, we use Backward Euler with Mid-point rule and set $h = 1/2^8$ so that we treat the solution as exact one. The results are shown in the right side of Figure~\ref{fig:ex56}. It can be observed that we take the logarithmic scale on the $y$-axis to observe the changes in the function value. When $\lambda+\kappa=0$, we found that solution value is convergent at a non-zero value, indicating that this setting does not belong to the absolute stability region $\mathcal{R}_\text{exact}$. However, when $\lambda + \kappa = -0.1$, we can see that the solution continues to monotonically decrease. We calculate its long-term behavior up to $t=1000$ and conclude that it converges to zero, although it does so slowly, indicating that this setting is within $\mathcal{R}_\text{exact}$. These phenomena show that numerically $\left\{(\lambda,k(\cdot)): \lambda+\int_{0}^{\infty} \Abs{k(\xi)} \diff \xi <0\right\}$ equals to $\mathcal{R}_\text{exact}$ when we use the type of kernel in \eqref{ex...6}. From this point of view, our modified definition is properly reasonable for at least some types of kernel function.



\bibliographystyle{siam}
\bibliography{ODEStability.bib}

\end{document}